\documentclass[12pt]{amsart}

\usepackage{latexsym}
\usepackage{a4wide}
\usepackage{amscd}
\usepackage{graphics}
\usepackage{amsmath}
\usepackage{amssymb}
\usepackage{mathrsfs}
\usepackage{amsthm}
\usepackage{bbm}
\usepackage{color}
\usepackage{accents}
\input xy
\xyoption{all}

\newcommand{\N}{\mathbb{N}}                     % the natural numbers
\newcommand{\Z}{\mathbb{Z}}                     % the integer numbers
\newcommand{\R}{\mathbb{R}}                     % the real line
\newcommand{\C}{\mathbb{C}}                     % the complex plane
\newcommand{\T}{\mathbb{T}}                     % the torus
\newcommand{\set}[2]{\left\{{#1}\mid{#2}\right\}}       % the set
\newcommand{\ind}{\mathrm{ind}}               % Fredholm index
\newcommand{\rank}{\mathrm{rank\,}}             % rank
\newcommand{\crit}{\mathrm{crit}}               % Critical set
\newtheorem{thm}{Theorem}[section]               % numbered absolutely
\newtheorem*{thm*}{Theorem}               % no number
\newtheorem{cor}[thm]{Corollary}        % numbered along with Theorem
\newtheorem*{cor*}{Corollary}        % no number
\newtheorem{lem}[thm]{Lemma}            % numbered along with Theorem
\newtheorem{prop}[thm]{Proposition}     % numbered along with Theorem
\theoremstyle{definition}
\newtheorem{rem}[thm]{Remark}           % numbered along with Theorem  
\title[Infinitely many periodic orbits of exact magnetic flows]{Infinitely many periodic orbits of exact magnetic flows\\ on surfaces for almost every subcritical energy level}

\author[Abbondandolo]{A. Abbondandolo}
\address{Ruhr Universit\"at Bochum, Fakult\"at f\"ur Mathematik\newline\indent Geb\"aude NA 4/33, D-44801 Bochum, Germany}
\email{alberto.abbondandolo@rub.de}

\author[Macarini]{L. Macarini}
\address{Universidade Federal do Rio de Janeiro, Instituto de Matem\'atica\newline\indent Cidade Universit\'aria, CEP 21941-909, Rio de Janeiro, Brazil}
\email{leonardo@impa.br}

\author[Mazzucchelli]{M. Mazzucchelli}
\address{CNRS, \'Ecole Normale Sup\'erieure de Lyon, Unit\'e de Math\'ematiques Pures et Appliqu\'ees\newline\indent  69364 Lyon Cedex 07, France}
\email{marco.mazzucchelli@ens-lyon.fr}

\author[Paternain]{G. P. Paternain}
\address{Department of Pure Mathematics and Mathematical Statistics\newline\indent  University of Cambridge, Cambridge CB3 0WB, UK}
\email{g.p.paternain@dpmms.cam.ac.uk}

\date{April 30, 2014}

\begin{document}

\renewcommand{\theenumi}{\roman{enumi}}
\renewcommand{\labelenumi}{(\theenumi)}

\begin{abstract}
We consider an exact magnetic flow on the tangent bundle of a closed surface. We prove that for almost every energy level $\kappa$ below the Ma\~n\'e critical value of the universal cover there are infinitely many periodic orbits with energy $\kappa$.

\tableofcontents
\end{abstract}

\maketitle

\section*{Introduction}

In this paper we study the existence of periodic orbits of prescribed energy of an exact magnetic flow on the tangent bundle of a closed surface.

Let $M$ be a closed surface. An exact magnetic flow on $TM$ is induced by the choice of a Riemannian metric $g$ and of a smooth 1-form $\theta$ on $M$. Let $Y:TM\to TM$ be the endomorphism uniquely defined by $d\theta=g(Y(\cdot),\cdot)$. A curve $\gamma:\R \rightarrow M$ is said to be a magnetic geodesic if it solves the  ODE
\begin{equation}
\label{mageq}
\nabla_t \dot{\gamma} = Y(\dot{\gamma}),
\end{equation}
where $\nabla_t$ denotes the Levi-Civita covariant derivative. The flow on $TM$ of this second order ODE is called an exact magnetic flow. The exactness refers to the fact that the 2-form $d\theta$ is exact. When $d\theta$ is replaced by a non-exact 2-form one talks about a non-exact magnetic flow. These flows are models for the motion of a particle in a magnetic field with Lorentz force $Y$ and were put into the context of modern dynamical systems by Arnold in~\cite{arn61b}.

Unlike the general non-exact case, exact magnetic flows admit a Lagrangian formulation: The equation (\ref{mageq}) is the Euler-Lagrange equation of the Lagrangian
\begin{align}\label{e:quadratic_Lagrangian}
L : TM \rightarrow \R, \qquad L(x,v) = \frac{1}{2} |v|_x^2 - \theta_x(v), \quad \forall (x,v)\in TM, 
\end{align}
where $|\cdot|_x$ denotes the norm on $T_x M$ which is induced by $g_x$. The energy which is associated to this autonomous Lagrangian is the function
\[
E : TM \rightarrow \R, \qquad E(x,v) = \frac{1}{2} |v|_x^2,  \quad \forall (x,v)\in TM.
\]
Therefore, the energy levels of this system are the same ones of the geodesic flow of $g$. However, the dynamics on these levels is quite different and changes with the value of the energy, because of the non-homogeneity of the Lagrangian. Roughly speaking, the magnetic flow on $E^{-1}(\kappa)$ behaves as a geodesic flow for high values of $\kappa$, where the dominant term in the Lagrangian is given by the metric $g$, and quite differently for low values of $\kappa$, for which the magnetic form $d\theta$ becomes dominant.

A critical value of the energy which marks a change in the dynamics and is relevant in this paper is the number
\[
c_u := - \inf \left\{ \frac{1}{T} \int_0^T L(\gamma,\dot{\gamma})\, dt \, \Big| \, \gamma: \R \rightarrow M, \mbox{ $T$-periodic contractible smooth curve}\right\},
\]
which is called {\em Ma\~n\'e critical value of the universal cover}. The specification ``of the universal cover'' refers to the fact that the infimum is taken over all contractible curves: By considering curves whose lift to other covering spaces of $M$ is closed, one finds other Ma\~n\'e critical values, which play different roles in the dynamical and geometric properties of the energy levels (see \cite{man97, cipp98}). The value $c_u$ is the smallest of all Ma\~n\'e critical values and, if  $d\theta\neq 0$, is strictly positive.

For an energy level $E^{-1}(\kappa)$ with $\kappa>c_u$ one can prove the same lower bounds for the number of periodic orbits which hold for closed Finsler geodesics. Indeed, if $\kappa$ is larger than $c_0$, i.e.~ the Ma\~{n}\'e critical value associated to the abelian cover of $M$, the flow on $E^{-1}(\kappa)$ is conjugated -- up to a time reparametrization -- to a Finsler geodesic flow. If $\kappa\in (c_u,c_0]$ this is not anymore true, but periodic orbits on $E^{-1}(\kappa)$ can be found as critical points of a functional which behaves essentially as the Finsler geodesic energy functional (see \cite{con06} and \cite{abb13}). In particular, when $M=S^2$ one has at least two periodic orbits on $E^{-1}(\kappa)$ for every $\kappa>c_u$, as proven for Finsler geodesic flows by Bangert and Long \cite{bl10}. Examples due to Katok \cite{kat73} show that in this case the number of periodic orbits on $E^{-1}(\kappa)$ can be exactly two.

In this paper we deal with the more mysterious range of energies $(0,c_u)$. When $\kappa$ is in this interval, $E^{-1}(\kappa)$ has always at least {\em one} periodic orbit. This result essentially goes back to the work of Taimanov  \cite{tai92b,tai92c,tai92}, but  was put into the context of Ma\~{n}\'e critical values and reproved using ideas from geometric measure theory by Contreras, Macarini and Paternain \cite{cmp04}. This result is the best one known so far concerning existence for {\em all} energy levels in the range $(0,c_u)$. For {\em almost every} $\kappa$ in $(0,c_u)$ the existence of at least {\em three} periodic orbits is known: The second periodic orbit was found by Contreras \cite{con06} (on manifolds of arbitrary dimension and for any Tonelli Lagrangian) and the third one by three of the authors of this paper in  \cite{amp13}.  In the latter paper the existence of {\em infinitely many} periodic orbits on $E^{-1}(\kappa)$ is also proved, under the assumption that $\kappa$ belongs to a suitable full-measure subset of $(0,c_u)$ and that all the periodic orbits on $E^{-1}(\kappa)$ are transversally non-degenerate. The aim of this paper is to remove all non-degeneracy assumptions and to prove the following: 

\begin{thm*}
Let $(M,g)$ be a closed Riemannian surface and let $\theta$ be a smooth 1-form. For almost every $\kappa$ in $(0,c_u)$, the exact magnetic flow induced by $g$ and $\theta$ has infinitely many periodic orbits on the energy hypersurface $E^{-1}(\kappa)$.
\end{thm*}

The proof of this theorem  is variational and is based on the study of the {\em free-period Lagrangian action functional}
\[
\mathbb{S}_{\kappa} (\gamma) := \int_0^T \bigl( L(\gamma,\dot\gamma) + \kappa \bigr)\, dt
\]
on the space of periodic curves $\gamma: \R \rightarrow M$ of arbitrary period $T$. In order to give a differentiable structure to the space of periodic curves with arbitrary period one identifies the $T$-periodic curve $\gamma$ with the pair $(x,T)$, where $x(s):= \gamma(sT)$ is 1-periodic. In this way one obtains the functional
\[
\mathbb{S}_{\kappa} : H^1(\T,M) \times (0,+\infty) \rightarrow \R, \qquad \mathbb{S}_{\kappa}(x,T) =  T \int_{\T} \bigl( L(x,\dot{x}/T) + \kappa \bigr)\, ds,
\]
on the Hilbert manifold $\mathcal{M}:= H^1(\T,M) \times (0,+\infty)$, where $\T:= \R/\Z$ and $H^1(\T,M)$ denotes the space of 1-periodic curves on $M$ of Sobolev class $H^1$. This functional is smooth and its critical points are exactly the periodic orbits of energy $\kappa$. The main difficulties in working with this functional are that for $\kappa< c_u$ it is unbounded from below on every connected component of $\mathcal{M}$ and fails to satisfy the Palais-Smale condition.

Up to lifting our Lagrangian system to the orientation double cover of $M$, we can assume that $M$ is orientable. As it is shown in \cite{amp13}, the periodic orbit $\alpha_{\kappa}$ which was found by Taimanov is a local minimizer of $\mathbb{S}_{\kappa}$ with $\mathbb{S}_{\kappa}(\alpha_{\kappa})<0$. The iterates $\alpha^n_{\kappa}$ of $\alpha_{\kappa}$ are also local minimizers and their $\mathbb{S}_{\kappa}$-action goes to $-\infty$  for $n\rightarrow \infty$. Using the fact that $\mathbb{S}_{\kappa}$ is unbounded from below on every connected component of $\mathcal{M}$, one would like to find other periodic orbits as mountain pass critical points associated to paths in $\mathcal{M}$ starting at some given iterate $\alpha^n_{\kappa}$ and ending at some curve with lower $\mathbb{S}_{\kappa}$-action. The main difficulty is the already mentioned lack of the Palais-Smale condition. However, thanks to the monotonicity of the function $\kappa \mapsto \mathbb{S}_{\kappa}$, one can overcome this difficulty by an argument due to Struwe \cite{str90}, which allows to find converging Palais-Smale sequences for almost every value of $\kappa$. This strategy was successfully implemented in \cite{amp13} and for a sufficiently high value of $n$ produces a mountain pass periodic orbit which is different from the one found by Contreras in \cite{con06}. When it is transversally non-degenerate, the $n$-th mountain pass periodic orbit can be shown to have positive mean Morse index. By this fact, in the non-degenerate case one can exclude that the infinitely many mountain pass critical points are the iterates of finitely many periodic orbits and gets the above mentioned result from \cite{amp13}.

If we drop the non-degeneracy assumption, mountain pass critical points might have mean Morse index zero and the above argument fails. In this paper we show that nevertheless the infinitely many mountain pass critical points  cannot be the iterates of finitely many periodic orbits. The reason is the following result, which we prove for manifolds $M$ of arbitrary dimension and for more general Lagrangians, since it might be of independent interest: A sufficiently high iterate of a periodic orbit cannot be a mountain pass critical point of the free-period action functional. See Theorem \ref{t:iterated_mountain_pass} below for the precise statement. Its proof uses the fact that periodic orbits with mean Morse index zero have Morse index zero: This is known to be true for the fixed period action functional, and in Section \ref{morseind} below we prove it for the free-period action functional. The proof of Theorem  \ref{t:iterated_mountain_pass} is contained in Section \ref{bangensec}: It is based on a careful analysis of the local behavior of $\mathbb{S}_{\kappa}$ near its critical set and on a homotopy argument due to Bangert \cite{ban80}.

There is a last technical difficulty which is dealt with in Section \ref{minmaxarg}: In order to overcome the lack of the Palais-Smale condition one needs to introduce a family of minimax  functions which depend monotonically on the energy $\kappa$ on some interval $I\subset (0,c_u)$. Moreover, the interval $I$ should remain the same also when considering the $n$-th minimax function. Since Taimanov's local minimizers $\alpha_{\kappa}$ might not depend continuously on $\kappa$, the obvious minimax class consisting of all paths joining $\alpha_{\kappa}^n$ to some curve with negative action might not produce a minimax value which is monotone with $\kappa$. This difficulty is overcome by considering classes of paths starting at an arbitrary point in the closure of all local minimizers of $\mathbb{S}_{\kappa}$ within a certain open subset of $\mathcal{M}$, and proving that in this case the required monotonicity holds on some interval which is independent of $n$.

We conclude this introduction by drawing a parallel between the main result of this paper and the waist theorem due to Bangert, which states that if a Riemannian metric on $S^2$ has a ``waist'', that is, a closed geodesic which is a local minimizer of the length functional, then it has infinitely many closed geodesics \cite[Theorem 4]{ban80}.  This statement has been incorporated by the later proved fact that every Riemannian metric on $S^2$ has infinitely many closed geodesics (see \cite{fra92b, ban93, hin93c}), but its proof remains interesting and is clearly related to our main result. In our theorem the existence of the ``waist''  does not have to be postulated, since it follows from the fact that the energy is below the Ma\~n\'e critical value and that we are dealing with a surface. Moreover, the topology of the surface does not play any role: In general we do not have any information on the homotopy class of Taimanov's local minimizer $\alpha_k$, but the free-period action functional is unbounded from below on every free homotopy class, so mountain pass critical values can always be defined.
Our Theorem \ref{t:iterated_mountain_pass}, which states that high iterates of periodic orbits cannot be mountain passes, is related to Theorem~2 in \cite{ban80}, which states essentially the same thing for closed geodesics on surfaces, although the two proofs are quite different. Actually, Theorem \ref{t:iterated_mountain_pass} (or its variant for fixed period problems), could be used to show that the waist theorem holds also for Finsler metrics on $S^2$ (see Remark \ref{waistrem} below for more details). Finally, unlike  the geodesic case, here we have to deal with the lack of the Palais-Smale condition, an issue which is ultimately responsible for the fact that the we can establish the existence of infinitely many periodic orbits only for almost every energy level.

\begin{footnotesize}
\subsection*{Acknowledgments.}
We are grateful to Marie-Claude Arnaud and Ludovic Rifford for a discussion about Franks' Lemma, and to Marie-Claude Arnaud and Jairo Bochi for suggesting the argument in the proof of Proposition~\ref{p:hyperbolic_perturbation}. A.~A.~is partially supported by the DFG-Grant AB 360/2-1. L. M. is partially supported by CNPq, Brazil. M.~M.~is partially supported by the  ANR projects WKBHJ (ANR-12-BLAN-WKBHJ) and  COSPIN (ANR-13-JS01-0008-01).

\end{footnotesize}

\section{The free-period action functional}\label{s:free_period_action}
\label{morseind}

\subsection{Definitions and basic properties}\label{s:basic_properties}
The results of sections~\ref{s:free_period_action} and~\ref{s:high_iterates_are_not_mountain_passes} hold for a closed manifold $M$ of arbitrary dimension  $d$ and for a general Lagrangian $L:TM \rightarrow \R$ of the form
\[
L(x,v) = \frac{1}{2} |v|_x^2 - \theta_x(v) - V(x),
\]
where $|\cdot|_{\cdot}$ denotes the norm which is induced by a Riemannian metric $g$ on $M$, $\theta$ is a smooth 1-form on $M$, and $V$ is a smooth real-valued function on $M$. The flow on $TM$ which is determined by the corresponding Euler-Lagrange equations preserves the energy function $E: TM \rightarrow \R$,
\[
E(x,v) = \partial_vL(x,v)[v] - L(x,v) = \frac{1}{2}|v|_x^2 + V(x).
\]
The space of closed curves on $M$ of Sobolev class $H^1$ and arbitrary period is identified with the Hilbert manifold
\[
\mathcal{M}:= H^1(\T,M) \times (0,+\infty),
\]
where $\T:= \R/\Z$ and the pair $(x,T)\in H^1(\T,M) \times (0,+\infty)$ corresponds to the $T$-periodic  curve $\gamma(t) = x(t/T)$. We denote the elements of $\mathcal{M}$ indifferently as $(x,T)$ or as $\gamma$. The free-period action functional
\[
\mathbb{S}: \mathcal{M} \rightarrow \R, \qquad \mathbb{S}(\gamma) = \mathbb{S}(x,T) := \int_0^T L(\gamma(t),\dot\gamma(t))\, dt = T \int_{\T} L(x(s),\dot{x}(s)/T)\, ds,
\]
is smooth on $\mathcal{M}$ and its critical points are precisely the periodic orbits $\gamma$ with energy $E(\gamma,\dot{\gamma})=0$. See \cite{con06} and \cite{abb13} for general facts about this functional.

The functional $\mathbb{S}$ is invariant with respect to the continuous action of $\T$ by time translations, which we denote by
\[
\T\times \mathcal{M} \rightarrow \mathcal{M}, \qquad (\tau,(x,T)) \mapsto \tau \cdot (x,T) := (x(\cdot+\tau),T).
\]
In particular, the critical set $\crit(\mathbb{S})$ of the free-period action functional consists of critical orbits $\T\cdot \gamma$.
We shall be interested in non-constant periodic orbits $\gamma=(x,T)$: In this case, $\T\cdot \gamma$ is a smooth embedded  circle in $\mathcal{M}$ (multiply covered by the action of $\T$, if $T$ is not the prime period of $\gamma$).

We endow $\mathcal{M}$ with the product metric between the standard Riemannian metric of $H^1(\T,M)$ and the standard metric of $(0,+\infty)\subset \R$, where the former metric is given by
\[
\langle u,v \rangle_{H^1} := \int_{\T} \bigl(g_x ( \nabla_s u,\nabla_s v) + g_x( u, v) \bigr)\, ds, \qquad \forall u, v \in T_x H^1(\T,M), \ x\in H^1(\T,M).
\] 
This metric induces a distance on $\mathcal{M}$ and a norm $\|\cdot\|$ on elements of $T\mathcal{M}$ and $T^* \mathcal{M}$.

We shall repeatedly use the fact that, when $M$ is compact, the functional $\mathbb{S}$ satisfies the Palais-Smale condition on subsets of $\mathcal{M}$ consisting of pairs $(x,T)$ for which $T$ is bounded and bounded away from 0: Every sequence $\gamma_h=(x_h,T_h)$ in $\mathcal{M}$ such that 
\[
0< \inf_h T_h \leq \sup_h T_h < +\infty,
\]
and
\[
\sup_h |\mathbb{S}(\gamma_h)| < +\infty, \qquad \|d\mathbb{S}(\gamma_h)\| \rightarrow 0,
\]
has a converging subsequence. See \cite[Proposition 3.12]{con06} or \cite[Lemma 5.3] {abb13}. In particular, $\|d\mathbb{S}\|$ is bounded away from zero on bounded closed subsets of $\mathcal{M}$ on which the second component $T$ is bounded away from zero and which contain no critical points.

\subsection{The Morse index and the iteration map}\label{s:Morse_index}
     
We denote by $\mathrm{ind}(\gamma) = \mathrm{ind}(x,T)$  the Morse index of the critical point $\gamma=(x,T)$ of $\mathbb{S}$. The Morse index of $x$ with respect to the fixed-period action $\mathbb{S}|_{H^1(\T,M) \times \{T\}}$ is denoted by $\mathrm{ind}_T(x)$. Clearly
\[
0 \leq \mathrm{ind}(x,T) - \mathrm{ind}_T(x) \leq 1.
\]
The precise relationship between these two indices in the special case of transversally non-degenerate critical points (as defined below) is discussed in \cite[Proposition 2.1]{amp13}.

Let $H:T^*M\to\R$ be the smooth Hamiltonian that is dual to $L$, i.e.
\[
H(x,p)=p(\mathfrak{L}^{-1}(x,p))-L(\mathfrak{L}^{-1}(x,p)),
\]
where $\mathfrak{L}:TM\to T^*M$ is the Legendre transform, that is, the diffeomorphism given by  $\mathfrak{L}(x,v)=(x,\partial_v L(x,v))$. We denote by $X_H$ the Hamiltonian vector field on the cotangent bundle $T^*M$ given by $\omega(X_H,\cdot)= -dH$, where $\omega=dp\wedge dq$ is the canonical symplectic structure of $T^*M$.  The Legendre transform conjugates the Euler-Lagrange flow of $L$ on any energy hypersurface $E^{-1}(\kappa)$ with the Hamiltonian flow $\phi_H^t$ of the vector field $X_H$ on the corresponding hypersurface $H^{-1}(\kappa)$.

Assume that the critical point $(x,T)$ corresponds to a non-constant periodic orbit $\gamma(t)=x(t/T)$. We set $z:=\mathfrak{L}(\gamma(0),\dot\gamma(0))$, and notice that the differential of the Hamiltonian flow $d\phi_H^T(z)$ preserves the coisotropic vector subspace $T_z(H^{-1}(0))$ and satisfies $d\phi_H^T(z)X_H(z)=X_H(z)$. Let $e_1,f_1,...,e_d,f_d$ be a symplectic basis of $T_zT^*M$ such that $f_1=X_H(z)$ and \[\mathrm{span}\{f_1,e_2,f_2,...,e_d,f_d\}=T_z (H^{-1}(0)).\]
In this symplectic basis, we can write $d\phi_H^T(z)$ as an element of $\mathrm{Sp}(2d)$ of the form
\[
d\phi_H^T(z)
=
\left(
\begin{array}{cc|ccc}
1 & 0 & 0 &\cdots & 0 \\
* & 1 & * & \cdots & * \\\hline
* & 0 \\
\vdots & \vdots & & P \\
* & 0
\end{array} \right),
\]
where $P\in\mathrm{Sp}(2d-2)$ is the linearized Poincar\'e map of $\phi_H^T$ at $z$, and the entries marked by $*$ contain some real numbers. Notice that the spectra of $d\phi_H^T(z)$ and $P$ are related by 
\begin{align}\label{e:spectra_of_linearized_flow_and_Poincaré_map}
\sigma(d\phi_H^T(z))=\{1\}\cup\sigma(P).
\end{align} 
The vector $(\dot{x},0)$ belongs to the kernel of $d^2 \mathbb{S}(x,T)$, and there is an isomorphism
\begin{equation}
\label{e:nullspace_of_the_Hessian}
\frac{\ker d^2 \mathbb{S}(x,T)}{\mathrm{span}\{(\dot{x},0)\}} \cong \ker (I-P),
\end{equation}
see \cite[Proposition A.3]{amp13}. The dimension of the  subspace $\ker (I-P)$ is by definition the transverse nullity $\mathrm{null}(x,T)$. When $\mathrm{null}(x,T)=0$ the orbit $\gamma$ is said to be transversally non-degenerate.

The multiplicative semigroup $\N=\{1,2,\dots,\}$ acts smoothly on $\mathcal{M}$ by iteration
\[
\N \times \mathcal{M} \rightarrow \mathcal{M}, \qquad (n,(x,T)) = (n,\gamma) \mapsto \gamma^n = \psi^n(x,T) := (x^n,nT),
\]
where $x^n(s):=x(ns)$. The free-period action functional $\mathbb{S}$ is equivariant with respect to this action on $\mathcal{M}$ and to the multiplicative action of $\N$ on $\R$, i.e.
\[
\mathbb{S}(\gamma^n) = n \,\mathbb{S}(\gamma), \qquad \forall (n,\gamma)\in \N \times \mathcal{M}.
\]
Notice that the iteration map $\psi^n:\mathcal{M}\hookrightarrow\mathcal{M}$ is a smooth embedding and maps the critical set of $\mathbb{S}$ into itself. The mean index of the critical point $\gamma=(x,T)$ is the non-negative real number
\[
\overline{\mathrm{ind}}(\gamma) = \overline{\mathrm{ind}}(x,T):= \lim_{n\rightarrow \infty} \frac{ \mathrm{ind}(\gamma^n)}{n} = \lim_{n\rightarrow \infty} \frac{ \mathrm{ind}_{nT}(x^n)}{n}.
\]
It is well known that $\overline{\mathrm{ind}}(x,T)$ vanishes if and only if the fixed-period Morse index $\mathrm{ind}_{nT}(x^n)$ vanishes for every $n\in \N$. The following result says that  the same fact is true for the free-period Morse index. 

\begin{prop}
\label{p:iterated_Morse_index}
The mean index $\overline{\mathrm{ind}}(\gamma)$ vanishes if and only if $\mathrm{ind}(\gamma^n)=0$ for every $n\in \N$.
\end{prop}

\begin{proof}
If $\mathrm{ind}(\gamma^n)=0$ for every $n$, the mean index $\overline{\mathrm{ind}}(\gamma)$ obviously vanishes. In order to prove the converse, we shall show that if $\ind(\gamma)>0$ then $\overline{\mathrm{ind}}(\gamma)>0$ as well. Since the mean index is homogeneous, i.e.\ $\overline{\mathrm{ind}}(\gamma^n)=n\,\overline{\mathrm{ind}}(\gamma)$, our proposition will readily follow.

Assume that $\ind(\gamma)>0$. The transversally non-degenerate case $\mathrm{null}(\gamma)=0$ was already treated in~\cite[Theorem~2.2]{amp13}. Therefore we are left to deal with the case $\mathrm{null}(\gamma)>0$.

Let $\mathcal{V}\subset C^\infty(M,\R)$ be the vector subspace of those  smooth functions $U:M\to\R$ whose 1-jet vanishes identically along $\gamma$, i.e.\ $U(\gamma(t))=0$ and $dU(\gamma(t))=0$ for all $t\in\R$. For $U\in\mathcal{V}$, we introduce the Lagrangian  \[(x,v)\mapsto L(x,v)-U(x),\] and we denote by $E_U:TM\to\R$ and $\mathbb{S}_U:\mathcal{M}\to\R$ the corresponding energy function and free-period action functional. A straightforward computation shows that $\gamma$ is still a solution of the Euler-Lagrange equations of $L-U$ with energy $E_U(\gamma,\dot\gamma)=0$. In particular $\gamma=(x,T)$ is still a critical point of the free-period action functional $\mathbb{S}_U$. The Hamiltonian dual to $L-U$ is precisely $H+U$, and we denote by $P_U\in\mathrm{Sp}(2d-2)$ the corresponding linearized Poincar\'e map at $z=\mathfrak{L}(\gamma(0),\dot\gamma(0))$. The so-called Hamiltonian Franks' Lemma (see \cite[Th.~1.2 and Sect.~3]{rr12}) implies that, for all open neighborhoods $\mathcal{U}\subset \mathcal{V}$ of the origin in the $C^2$ topology, the image of the map $\mathcal{U}\to\mathrm{Sp}(2d-2)$ given by $U\mapsto P_U$ contains an open neighborhood of $P$.

Consider the splitting $\R^{2d-2}=V\oplus W$, where $V$ is the symplectic subspace invariant by $P$ whose complexification is the generalized eigenspace of the eigenvalue 1,  and $W$ is its symplectic orthogonal (whose complexification is the direct sum of the generalized eigenspaces of the remaining eigenvalues). Using this splitting, we can write $P$ as a block-diagonal matrix of the form
\[
P = \left(
\begin{matrix}
A & 0 \\
0 & B
\end{matrix} \right),
\]
where $A=P|_{V}$ and $B=P|_{W}$. Notice that the spectra of $A$ and $B$ satisfy $\sigma(A)=\{1\}$  and $\sigma(B)\cap\{1\}=\varnothing$. By Proposition~\ref{p:hyperbolic_perturbation} in the Appendix, the matrix $A$ belongs to the closure of the set of symplectic matrices whose spectrum does not intersect the unit circle. Thus, we can find an arbitrarily $C^2$-small smooth potential $U:M\to\R$ such that $P_U$ has the form
\[
P_U = \left(
\begin{matrix}
A' & 0 \\
0 & B
\end{matrix} \right),
\]
and the spectrum  $\sigma(A')$ does not intersect the unit circle. In particular
\begin{align}\label{e:relation_between_spectra_after_perturbation}
 S^1\cap\sigma(P_U)=S^1\cap\sigma(P)\setminus\{1\}.
\end{align}
We recall that the spectra of $d\phi_{H+U}^T(z)$ and $P_U$ are related as in~\eqref{e:spectra_of_linearized_flow_and_Poincaré_map}.

For all $s\in[0,1]$, we denote by $\Lambda_s:S^1\to\N$ Bott's function associated to the linearization along $\gamma=(x,T)$ of the $d$-dimensional second order system which is induced by the perturbed Lagrangian $L-sU$. We refer the reader to Bott's original paper \cite{bot56},  or to \cite[ch.~9]{lon02} and \cite[sect.~2.2]{maz11}, for the definition and main properties of Bott's functions. Here, we just recall that these functions compute in particular the fixed-period Morse indices as $\ind_T(x)=\Lambda_0(1)$, and the mean index as the average
\[
\overline{\ind}(\gamma)=\frac{1}{2\pi}\int_{0}^{2\pi} \Lambda_0(e^{i\theta})\,d\theta.
\]
Each Bott's function $\Lambda_s$ is lower semi-continuous and  is actually  locally  constant outside the intersection $\sigma(d\phi_{H+sU}^T(z))\cap S^1$. Moreover, the function $s\mapsto \Lambda_s(e^{i\theta})$ is constant provided \[e^{i\theta}\not\in\bigcup_{s\in[0,1]}\sigma(d\phi_{H+sU}^T(z)).\]

We choose $\alpha>0$ small enough  such that the spectrum of $d\phi_{H}^T(z)$ does not contain any eigenvalue on the unit circle with argument in the interval $(0,\alpha]$, i.e.
\begin{align*}
\big\{e^{i\theta}\ \big|\ \theta\in(0,\alpha]\big\}\cap\sigma(d\phi_{H}^T(z))=\varnothing.
\end{align*}
This implies that the function $\theta\mapsto\Lambda_0(e^{i\theta})$ is constant on the interval $(0,\alpha]$. By~\eqref{e:relation_between_spectra_after_perturbation} and~\eqref{e:spectra_of_linearized_flow_and_Poincaré_map}, we also have
\begin{gather}
\label{e:eigenvalue_close_to_1} \big\{e^{i\theta}\ \big|\ \theta\in(0,\alpha]\big\}\cap\sigma(d\phi_{H+U}^T(z))=\varnothing.
\end{gather}
We denote by $\ind(\mathbb{S}_{U},\gamma)$ the Morse index of the free-period action functional $\mathbb{S}_{U}$ at $\gamma$. Since our $U$ can be chosen to be arbitrarily $C^2$-small, we can assume that 
\begin{gather}
\label{e:eigenvalue_close_to_2} e^{ i\alpha}\not\in\bigcup_{s\in[0,1]}\sigma(d\phi_{H+sU}^T(z)).
\end{gather}
Moreover, by the lower semi-continuity of the Morse index, we can further assume that 
\[
\ind(\mathbb{S}_{U},\gamma)\geq\ind(\gamma)>0.
\]
Notice that $\gamma$ is transversally non-degenerate for the free-period action functional $\mathbb{S}_{U}$. Therefore, by~\cite[Theorem~2.2]{amp13}, we infer $\Lambda_1(1)>0$. By~\eqref{e:eigenvalue_close_to_1}, the function $\theta\mapsto\Lambda_1(e^{i\theta})$ is constant on the interval $(0,\alpha]$, and since Bott's functions are lower semi-continuous, this constant must be larger than or equal to $\Lambda_1(1)$.  By~\eqref{e:eigenvalue_close_to_2} and the above mentioned continuity property of homotopies of Bott's functions, we have
\[
\Lambda_s(e^{i\alpha})=\Lambda_1(e^{i\alpha})\geq\Lambda_1(1)>0,\qquad\forall s\in[0,1].
\]
This implies that 
\[
\overline{\ind}(\gamma)
=\frac{1}{2\pi} \int_{0}^{2\pi} \Lambda_0(e^{i\theta})\,d\theta
\geq
\frac{1}{2\pi} \int_{0}^\alpha \Lambda_0(e^{i\theta})\,d\theta
=
\frac{\alpha}{2\pi} \Lambda_0(e^{i\alpha})
\geq
\frac{\alpha}{2\pi} \Lambda_1(1)>0. \qedhere
\] 
\end{proof}

We conclude this section with the following lemma, which turns out to be useful when dealing with degenerate critical points. The lemma was first proved by Gromoll and Meyer \cite[Lemma~2]{gm69b} in the context of non-magnetic closed geodesics. We include its short proof here for the reader's convenience.

\begin{lem}\label{l:iterated_nullity}
Let $\gamma=(x,T)$ be a critical point of $\mathbb{S}$ which corresponds to a non-constant periodic orbit.
Then there is a partition $\N=\N_1\cup...\cup\N_k$, integers $n_1\in\N_1$, ..., $n_k\in\N_k$,  and $\nu_1,...,\nu_k\in\{0,...,2d-2\}$ with the following property: $n_j$ divides all the integers in $\N_j$, and $\mathrm{null}(\gamma^n)=\nu_j$ for all $n\in\N_j$. 
\end{lem}

\begin{proof}
Since $P$ is an automorphism, the geometric multiplicity of the eigenvalue 1 varies under iteration according to
\begin{align}\label{e:Bott}
\dim\ker(P^n-\mathrm{id})=\sum_{\lambda\in\sqrt[n]{1}}\dim_{\C}\ker_{\C}(P-\lambda\,\mathrm{id}).
\end{align}
Let $\sigma(P)$ be the set of eigenvalues of $P$, and, for all $n\in\N$, we set
\[
\sigma_n(P)=\{\lambda\in\sigma(P)\ |\ \lambda^n=1\}.
\]
We define an equivalence relation on $\N$, by saying that $m\sim n$ when $\sigma_m(P)=\sigma_n(P)$. We denote by $\N_1,...,\N_r$ the equivalence classes, and by $n_j$ the minimum of $\N_j$. If 
\[\sigma_{n_j}(P)=\{\exp(i2\pi p_1/q_1),...,\exp(i2\pi p_n/q_n)\},\] 
where $p_h$ and $q_h$ are relatively prime, the integers in $\N_j$ are common multiples of $q_1,...,q_n$, and $n_j$ is the least common multiple of $q_1,...,q_n$. By~\eqref{e:nullspace_of_the_Hessian} and~\eqref{e:Bott} we have
\[
\mathrm{null}(\gamma^n)=\sum_{\lambda\in\sigma_n(P)}\dim_{\C}\ker_{\C}(P-\lambda\,\mathrm{id}),
\]
which is the same integer for all $n$ belonging to the same set $\N_j$.
\end{proof}

\section{High iterates of periodic orbits are not mountain passes}\label{s:high_iterates_are_not_mountain_passes}
\label{bangensec}

\subsection{A tubular neighborhood lemma for critical orbits}

In this subsection, we will denote by $\nabla\mathbb{S}$ the gradient vector field of the free-period action functional $\mathbb{S}$ with respect to the following Riemannian metric on $\mathcal{M}$
\begin{equation}
\label{e:Riemannian_metric}
\begin{split}
\langle ( u,R),(v,S) \rangle_{(x,T)}
&:= \frac{RS}{T} + T \int_0^1  \Bigl( g_x(u, v)  + \frac{1}{T^2} \, g_x (\nabla_s u, \nabla_s v)  \Bigr)\, 
d s \\ &= \frac{RS}{T} + \int_0^T \bigl( g_{\gamma}(\xi,\eta) + g_{\gamma} (\nabla_t \xi, \nabla_t \eta) \bigr)\, dt,
\end{split} 
\end{equation}
where $\xi(t)=u(t/T)$ and $\eta(t)=v(t/T)$. This metric is equivalent to the standard one of $\mathcal{M}$ (see Section~\ref{s:basic_properties}) on subsets consisting of pairs $(x,T)$ for which $T$ is bounded and bounded away from $0$. In particular, in these subsets the free-period action functional $\mathbb{S}$ satisfies the Palais-Smale condition with respect to the metric (\ref{e:Riemannian_metric}). 

The advantage of this metric is that it makes the iteration map $\psi^n$ conformal with constant conformal factor $n$. Indeed, using the identity $d\psi^n(x,T)[(u,R)] = (u^n,nR)$, one checks easily that
\[
\langle d\psi^n(x,T) [(u,R)] , d\psi^n(x,T) [(v,S)] \rangle_{(x^n,nT)} = n \, \langle ( u,R),(v,S) \rangle_{(x,T)},
\]
for every $(u,R)$ and $(v,S)$ in the tangent space of $\mathcal{M}$ at $(x,T)$. In particular, if the superscript $*$ denotes the adjoint operation with respect to the metric (\ref{e:Riemannian_metric}), we have that 
\begin{equation}
\label{proiettore}
d\psi^n(x,T) \, d\psi^n(x,T)^* = n \, \Pi,
\end{equation}
where $\Pi$ is the orthogonal projector onto the image of $d\psi^n(x,T)$, that is, the tangent space of the submanifold $\psi^n(\mathcal{M})$ at $(x^n,nT)$. Furthermore, it is easy to see that the orthogonal space of $T_{(x^n,nT)} \psi^n(\mathcal{M})$ is
\begin{equation}
\label{ortho}
\bigl( T_{(x^n,nT)} \psi^n(\mathcal{M}) \bigr)^{\perp} = \left\{ (u,R)\in T_{(x^n,nT)} \mathcal{M} \, \bigg| \, R=0,\ \sum_{j=0}^{n-1} u\big( \tfrac{s+j}{n} \big) = 0 \; \forall s\in \T \right\}.
\end{equation}
The above facts allow us to prove the following: 

\begin{lem}
If $\nabla$ denotes the gradient with respect to the metric (\ref{e:Riemannian_metric}), then
\begin{equation}
\label{e:psim_commutes_with_nablaS}
 \nabla \mathbb{S}(\gamma^n) = d \psi^n (\gamma) \nabla \mathbb{S}(\gamma),\qquad\forall \gamma\in\mathcal{M}.
\end{equation}
\end{lem}

\begin{proof}
By applying the projector $\Pi$ onto the tangent space of $\psi^n(\mathcal{M})$ at $\gamma^n=\psi^n(\gamma)$ and its complementary projector $I-\Pi$, the identity (\ref{e:psim_commutes_with_nablaS}) is equivalent to the two identities
\begin{gather}
\label{eq1}
\Pi  \, \nabla \mathbb{S}(\gamma^n) = d\psi^n(\gamma) \nabla \mathbb{S} (\gamma), \\
\label{eq2}
(I-\Pi)  \nabla \mathbb{S}(\gamma^n) = 0.
\end{gather}
By differentiating the identity $\mathbb{S}(\psi^n(\gamma)) = n \, \mathbb{S}(\gamma)$ we find the formula
\[
d\psi^n(\gamma)^* \nabla \mathbb{S} (\psi^n(\gamma)) = n \nabla  \mathbb{S}(\gamma).
\]
If we apply $d\psi^n(\gamma)$ to both sides and we use (\ref{proiettore}), we see that (\ref{eq1}) holds. The identity (\ref{eq2}) says that the differential of $\mathbb{S}$ at $\gamma^n$ vanishes on the orthogonal space of $T_{\gamma^n} \psi^n(\mathcal{M})$. By the identity (\ref{ortho}) an element of this orthogonal space has the form $(u,0)$, where
\[
\sum_{j=0}^{n-1} u \big(  \tfrac{s+j}{n} \big) = 0, \qquad \forall s\in \T.
\]
By setting $\xi(t) := u(t/(nT))$ we have
\begin{equation}
\label{vanish}
\sum_{j=0}^{n-1} \xi(t+jT) = 0, \qquad \forall t\in \R.
\end{equation}
Therefore,
\[
d\mathbb{S}(\gamma^n)[(u,0)] = \int_0^{nT} \Bigl( \partial_x L(\gamma,\dot\gamma)[\xi] + \partial_v L(\gamma,\dot\gamma)[\nabla_t \xi] \Bigr)\, dt
\]
vanishes, because the functions $t\mapsto  \partial_x L(\gamma(t),\dot\gamma(t))$ and $t\mapsto  \partial_v L(\gamma(t),\dot\gamma(t))$ are $T$-periodic while $\xi$ satisfies (\ref{vanish}).
\end{proof}

The following tubular neighborhood lemma, which is reminiscent of some arguments in~\cite[Section~3]{gm69b}, will be useful later on, in the proof of Theorem~\ref{t:iterated_mountain_pass}.

\begin{lem}\label{l:Gromoll-Meyer}
Let $\gamma$ be a critical point of the free-period action function $\mathbb{S}$ such that, for some  $n\in\N$, we have $\mathrm{ind}(\gamma)=\mathrm{ind}(\gamma^n)$ and $\mathrm{null}(\gamma)=\mathrm{null}(\gamma^n)$. Then, for every sufficiently small open neighborhood $\mathcal{U}$ of $\T\cdot \gamma$ there is an open neighborhood $\mathcal{V}$ of the submanifold $\psi^n(\mathcal{U})$ and a smooth map
\[
r :[0,1] \times \mathcal{V} \rightarrow \mathcal{V}, \quad (t,\beta) \mapsto r_t(\beta),
\]
such that:
\begin{enumerate}
\item $r_0=\mathrm{id}$;
\item $r_t|_{\psi^n(\mathcal{U})}= \mathrm{id}$ for every $t\in [0,1]$;
\item $r_1(\mathcal{V}) = \psi^n(\mathcal{U})$;
\item $\tfrac{d }{d  t} \mathbb{S}(r_t(\beta))< 0$ for all $\beta \in \mathcal{V}\setminus \psi^n(\mathcal{U})$.
\end{enumerate}
\end{lem}

\begin{proof}
 Let $\mathcal{U}$ be an open neighborhood of the critical orbit $\T\cdot \gamma$ (which may consist of an embedded circle or just the critical point alone if $\gamma$ is a stationary curve). The image of this neighborhood under the iteration map $\psi^n$ has a normal bundle $ N (\psi^n(\mathcal{U}))$. By our assumptions on the Morse index and nullity, and since $0$ is an isolated point in the spectrum of $d ^2\mathbb{S}(\gamma^n)$, there exists a constant $\delta>0$ such that
\[
d ^2\mathbb{S}(\gamma^n)[\xi,\xi]\geq\delta \langle \xi,\xi\rangle_{\gamma^n} ,\qquad\forall \xi\in N_{\gamma^n}(\psi^n(\mathcal{U})).
\]
Since the free-period action functional $\mathbb{S}$ is smooth, up to shrinking the neighborhood $\mathcal{U}$ and choosing a smaller constant $\delta>0$, we can assume that
\begin{align}\label{e:pos_def_along_fibers}
d ^2 \mathbb{S}(\beta)[\eta,\eta]\geq\delta\langle \eta,\eta \rangle_{\beta},\qquad\forall \beta \in \psi^n(\mathcal{U}),\ \eta \in N_{\beta }(\psi^n(\mathcal{U})).
\end{align}
We now consider a sufficiently small neighborhood of the  0-section and apply to it the exponential map of the metric~\eqref{e:Riemannian_metric}. We obtain a tubular neighborhood $\mathcal{V}$ of $\psi^n(\mathcal{U})$ with associated deformation retraction $r_t:\mathcal{V}\to \mathcal{V}$ given by 
\[
r_t\circ\exp_\beta (\eta)=\exp_\beta ((1-t)\eta),\qquad\forall \beta \in \psi^n(\mathcal{U}),\ \eta \in N_{\beta }(\psi^n(\mathcal{U})),
\]
which satisfies (i), (ii) and (iii). Notice that, given $\beta \in \psi^n(\mathcal{U})$, the tangent space to the fiber of $r_1$ at $\beta$ is $ T_\beta (r_1^{-1}(\beta))= N_\beta (\psi^n(\mathcal{U}))$. By~\eqref{e:psim_commutes_with_nablaS}, the restriction of the free-period action functional $\mathbb{S}$ to each fiber $r_1^{-1}(\beta)$ has a critical point at $\beta $. Moreover, by~\eqref{e:pos_def_along_fibers}, such a restriction has a positive definite Hessian. In particular it is a convex function near its local minimum $\beta $, and up to shrinking the tubular neighborhood $\mathcal{V}$ we have $\tfrac{d }{d  t} \mathbb{S}(r_t(\beta))< 0$ for all $\beta \in \mathcal{V}\setminus \psi^n(\mathcal{U})$. Therefore, also (iv) holds.
\end{proof}

\subsection{Resolution of degenerate critical circles}

Let $\T\cdot \gamma$ be an isolated critical circle of the free-period action functional $\mathbb{S}$, and set
\begin{align*}
\iota&:=\mathrm{ind}(\gamma),\\
\nu&:=\mathrm{null}(\gamma)+1=\dim\ker d^2\mathbb{S}(\gamma).
\end{align*}
We denote by $\mathcal{Z}\to \T\cdot\gamma$ the vector bundle whose fiber over any $t\cdot\gamma$ is equal to the intersection of the normal bundle of the critical circle $\T\cdot\gamma\subset\mathcal{M}$ with the kernel of the Hessian of $\mathbb{S}$, i.e.
\[
\mathcal{Z}_{t\cdot\gamma}
=
N_{t\cdot\gamma}(\T\cdot\gamma)\cap\ker d^2\mathbb{S}(t\cdot\gamma).
\]
The rank of  $\mathcal{Z}$ is finite, and it is actually equal to the nullity $\nu-1$. We denote by $\mathcal{Q}$ the orthogonal complement of the subbundle $\mathcal{Z}\subset N(\T\cdot\gamma)$ with respect to the Riemannian metric of $\mathcal{M}$ (see Section~\ref{s:basic_properties}). By the very definition of $\mathcal{Q}$, and since 0 is an isolated point in the spectrum of $d^2\mathbb{S}(t\cdot\gamma)$, there exists a constant $\delta>0$ such that 
\begin{align}\label{e:non_degenerate_Hessian}
\|d^2\mathbb{S}(t\cdot\gamma)[\xi,\cdot]\|\geq\delta\|\xi\|,\qquad
\forall t\in\T,\ \xi\in\mathcal{Q}_{t\cdot\gamma},
\end{align}
where the norm on the left-hand side of the above inequality is the dual of the Riemannian one on $\mathcal{M}$, and $\mathcal{Q}_{t\cdot\gamma}\subset\mathcal{Q}$ is the fiber above $t\cdot\gamma$. Since the orthogonal group of the infinite-dimensional separable Hilbert space $\mathbb{E}$ is connected, the vector bundle $\mathcal{Q}$ is trivial. Therefore we can make the identification 
\[ N (\T\cdot \gamma)\cong \mathcal{Z}\times\mathbb{E}.\] 
We will denote the coordinates in $\mathcal{Z}\times\mathbb{E}$ as $(t,z,y)$, where $(t,z)\in\mathcal{Z}$ and  $y\in\mathbb{E}$.

Let $\mathcal{U}_R\subset\mathcal{Z}$ denote the open neighborhood of radius $R$ of the 0-section, and $\mathcal{B}_R\subset\mathbb{E}$ be the open balls of radius $R$ centered at the origin. For a sufficiently small  $R>0$, the exponential map of the normal bundle identifies  $\overline{\mathcal{U}}_R \times\overline{\mathcal{B}}_R$ with a neighborhood of $\T\cdot \gamma$ that does not contain other critical points of $\mathbb{S}$. Thus, from now on we can see the free-period action functional as being of the form 
\[
\mathbb{S}:\overline{\mathcal{U}}_R \times\overline{\mathcal{B}}_R\to\R,
\] 
with $\crit(\mathbb{S})=\{0\mbox{-section}\}\times\{0\}$. We equip $\overline{\mathcal{U}}_R$ with the Riemannian metric pulled-back from the one of $\mathcal{M}$ by means of the exponential map and equip $\overline{\mathcal{U}}_R \times\overline{\mathcal{B}}_R$ with the product Riemannian metric induced by the one on $\overline{\mathcal{U}}_R$ and the flat Hilbert metric on $\mathbb{E}$. The standard Riemannian metric on $\mathcal{M}$ is uniformly equivalent on $\overline{\mathcal{U}}_R \times\overline{\mathcal{B}}_R$ to the product one. Therefore, since $\mathbb{S}$ satisfies the Palais-Smale condition inside  $\overline{\mathcal{U}}_R \times\overline{\mathcal{B}}_R$ with respect to the Riemannian metric of $\mathcal{M}$, it  does it with respect to the product metric as well. 

We shall show that we can resolve the degeneracy of the isolated critical circle with a small perturbation of the function supported on any neighborhood of the circle. The following result is a version of Marino and Prodi's perturbation lemma from \cite{mp75}, the difference being that we start from an isolated critical circle rather than an isolated critical point.

\begin{lem}\label{l:Marino-Prodi}
For any  neighborhood $\mathcal{U}$ of the critical circle $\{0\mbox{-section}\}\times\{0\}$ there exists a smooth function $\mathbb{S}':\overline{\mathcal{U}}_R \times\overline{\mathcal{B}}_R\to\R$ that satisfies the Palais-Smale condition, possesses only finitely many critical points, all of which are non-degenerate with Morse index larger than or equal to the original one $\iota$, and such that the difference $\mathbb{S}'-\mathbb{S}$ is supported in $\mathcal{U}$ and arbitrarily $C^2$-small.
\end{lem}

\begin{proof}
Up to choosing a sufficiently small constant $\delta_1>0$, the inequality~\eqref{e:non_degenerate_Hessian} can be rephrased in our coordinates as
\begin{align*}
\|\partial^2_{yy}\mathbb{S}(t,0,0)[v,\cdot]\|\geq\delta_1\|v\|,\qquad
\forall t\in\T,\ v\in\mathbb{E}.
\end{align*}
By the implicit function theorem, there exist $r_1,r_2\in(0,R/2)$ and a smooth map 
\[\psi: \overline{\mathcal{U}}_{r_1}\to\overline{\mathcal{B}}_{r_2}\] 
such that, for all $(t,z,y)\in\overline{\mathcal{U}}_{r_1}\times\overline{\mathcal{B}}_{r_2}$, we have $\partial_y\mathbb{S}(t,z,y)=0$ if and only if $y=\psi(t,z)$. We define 
\[\Psi:\overline{\mathcal{U}}_{r_1}\times\overline{\mathcal{B}}_{r_2}\to \overline{\mathcal{U}}_{R}\times\overline{\mathcal{B}}_{R}\]
by $\Psi(t,z,y)=(t,z,\psi(t,z)+y)$.  Notice that $\Psi$ is a diffeomorphism onto a neighborhood of the critical circle $\{0\mbox{-section}\}\times\{0\}$. From now on we will employ the new coordinates defined by $\Psi$, and therefore we will simply write $\mathbb{S}:\overline{\mathcal{U}}_{r_1}\times\overline{\mathcal{B}}_{r_2}\to\R$ for the composition $\mathbb{S}\circ\Psi$. If we consider $\overline{\mathcal{U}}_{r_1}\times\overline{\mathcal{B}}_{r_2}$ as a trivial bundle over $\overline{\mathcal{U}}_{r_1}$ by means of the projection onto the first  factor, the function $\mathbb{S}$ has non-degenerate fiberwise critical points everywhere on the  0-section. Namely, up to reducing $\delta_1>0$ and $r_1>0$, we have
\begin{align}\label{e:fiberwise_nondegenerate_critical_points}
\partial_y\mathbb{S}(t,z,0)=0, \quad
\|\partial^2_{yy}\mathbb{S}(t,z,0)[v,\cdot]\|\geq\delta_1\|v\|,\qquad
\forall (t,z)\in \overline{\mathcal{U}}_{r_1},\ v\in\mathbb{E}.
\end{align}
In particular, for all $r\in(0,\min\{r_1,r_2\})$ sufficiently small, there exists $\delta_2>0$ such that
\begin{align*}
\|\partial_y\mathbb{S}(t,z,y)\| \geq\delta_2,\qquad
\forall (t,z,y)\in  \overline{\mathcal{U}}_{r}\times\big(\overline{\mathcal{B}_{r}\setminus \mathcal{B}_{r/2}}\big).
\end{align*}
We denote by $F:\overline{\mathcal{U}}_{r}\to\R$ the smooth function $F(t,y)=\mathbb{S}(t,y,0)$. Since the critical circle $\{0\mbox{-section}\}\times\{0\}$ is the only critical point of $\mathbb{S}:\overline{\mathcal{U}}_R\times \overline{\mathcal{B}}_R\to\R$, the 0-section is precisely the critical point set of $F$. In particular, there exists $\delta_3>0$ such that
\begin{align*}
\|\nabla F(t,z)\| \geq\delta_3,\qquad
\forall (t,z)\in  \overline{\mathcal{U}_{r}\setminus \mathcal{U}_{r/2}}.
\end{align*}
We recall that $\overline{\mathcal{U}}_{r}$ is a compact manifold of dimension $\nu$ with smooth boundary. By the density of the Morse functions into the space of smooth functions on finite-dimensional manifolds (see e.g.~\cite[Corollary~6.8]{mil63}), for any  $\epsilon>0$ we can find a Morse function $F_\epsilon:\overline{\mathcal{U}}_{r}\to\R$ that is $\epsilon$-close to $F$ in the $C^2$-topology. Notice that, if $\epsilon<\delta_3$, the function $F_\epsilon$ has finitely many non-degenerate critical points, all of which are contained in $\mathcal{U}_{r/2}$. Let $\chi:\mathcal{U}_r\to[0,1]$ be a compactly supported smooth function that is identically equal to $1$ on $\mathcal{U}_{r/2}$, and let $\rho:\mathcal{B}_r\to[0,1]$ be a compactly supported smooth function that is identically equal to $1$ on $\mathcal{B}_{r/2}$. We define the function $\mathbb{S}'$ of the lemma by 
\[\mathbb{S}'(t,z,y)=\mathbb{S}(t,z,y) + \chi(t,z)\rho(y) (F_\epsilon(t,z)-F(t,z)).\]
Let us verify that, for $\epsilon>0$ sufficiently small, $\mathbb{S}'$ satisfies the desired properties. First of all, $\mathbb{S}'$ tends to $\mathbb{S}$ in the $C^2$-topology as $\epsilon\to0$, and is equal to $\mathbb{S}$ outside $\mathcal{U}_{r}\times\mathcal{B}_{r}$. Since $\mathbb{S}:\overline{\mathcal{U}}_R\times\overline{\mathcal{B}}_R\to\R$ satisfies the Palais-Smale condition, there exists $\delta_4>0$ such that
\[
\|\nabla\mathbb{S}(t,z,y)\|\geq\delta_4,\qquad\forall (t,z,y)\in \overline{
(\mathcal{U}_r\times\mathcal{B}_r) 
\setminus
(\mathcal{U}_{r/2}\times\mathcal{B}_{r/2}) 
}.
\]
In the region $\overline{
(\mathcal{U}_r\times\mathcal{B}_r) 
\setminus
(\mathcal{U}_{r/2}\times\mathcal{B}_{r/2})}$ we have
\begin{align*}
\|\nabla\mathbb{S}'\|
&=
\|\nabla\mathbb{S}+ (F_\epsilon-F)(\rho\,\nabla\chi + \chi\,\nabla\rho)+ \chi\,\rho\,(\nabla F_\epsilon-\nabla F)\|\\
&\geq
\|\nabla\mathbb{S}\|-|F_\epsilon-F|(\|\nabla\chi\| + \|\nabla\rho\|)
-
\|\nabla F_\epsilon-\nabla F\|\\
&\geq 
\delta_4-\epsilon(\|\nabla\chi\| + \|\nabla\rho\|)-\epsilon\\
&\geq 
\delta_4/2
\end{align*}
provided 
\[\epsilon\leq\frac{\delta_4}{2(1+\max\|\nabla\chi\|+\max\|\nabla\rho\|)}.\] 
In particular, $\mathbb{S}'$ has no critical points and satisfies the Palais-Smale condition in this region. In the region $\overline{\mathcal{U}}_{r/2}\times\overline{\mathcal{B}}_{r/2}$, the function $(t,z,y)\mapsto\chi(t,z)\rho(y)$ is identically equal to $1$, and therefore $\partial_y\mathbb{S}'=\partial_y\mathbb{S}$. This, together with~\eqref{e:fiberwise_nondegenerate_critical_points}, shows that all the critical points of $\mathbb{S}'$ are contained in $\overline{\mathcal{U}}_{r/2}\times\{0\}$. All such critical points  are non-degenerate with Morse index larger than or equal to $\iota$, since they are non-degenerate for the restricted function $\mathbb{S}'|_{\mathcal{U}_{r/2}\times\{0\}}=F_\epsilon$ and since, by the inequality in~\eqref{e:fiberwise_nondegenerate_critical_points}, they are fiberwise non-degenerate and thus have fiberwise Morse index equal to $\iota$. In particular, $\mathbb{S'}$ has finitely many critical points in $\overline{\mathcal{U}}_{r/2}\times\overline{\mathcal{B}}_{r/2}$. Finally, suppose that $\{(t_n,z_n,y_n)\}$ is a Palais-Smale sequence for $\mathbb{S}'$ contained in $\overline{\mathcal{U}}_{r/2}\times\overline{\mathcal{B}}_{r/2}$. This implies $\partial_y\mathbb{S}'(t_n,z_n,y_n)=\partial_y\mathbb{S}(t_n,z_n,y_n)\to0$, and by~\eqref{e:fiberwise_nondegenerate_critical_points} we have that $y_n\to0$. Since the sequence $\{(x_n,z_n)\}$ varies inside the compact set $\overline{\mathcal{U}}_{r/2}$, the Palais-Smale sequence admits a converging subsequence. This proves that $\mathbb{S}'$ satisfies the Palais-Smale condition inside $\overline{\mathcal{U}}_{r/2}\times\overline{\mathcal{B}}_{r/2}$.
\end{proof}

\subsection{Properties of sublevel sets near  critical circles}
Now,  we trivialize the normal bundle $N (\T\cdot \gamma)$ (once again, we recall that this is possible since the orthogonal group of the infinite-dimensional separable Hilbert space $\mathbb{E}$ is connected). Thus, we make the identification $ N (\T\cdot \gamma)\cong \T\times \mathbb{E}$, and for a sufficiently small $R>0$, we  employ the exponential map in order to identify  $\T\times \overline{\mathcal{B}}_R\subset \T\times \mathbb{E}$ with a neighborhood of $\T\cdot \gamma$ which does not contain other critical points of $\mathbb{S}$. The restriction of the free-period action functional to this neighborhood will be a smooth function of the form $\mathbb{S}:\T\times \overline{\mathcal{B}}_R\to\R$. We equip $\T\times\overline{\mathcal{B}}_R$ with the product Riemannian metric induced by the Euclidean metric on $\T$ and the flat Hilbert metric on $\mathbb{E}$. Since this metric is uniformly equivalent to the standard one of $\mathcal{M}$, the free-period action functional $\mathbb{S}$ still satisfies the Palais-Smale condition with respect to it. We denote by $\nabla \mathbb{S}$ the gradient of the free-period action functional $\mathbb{S}$ with respect to the product metric, and by $\phi_s$ the associated negative gradient flow. We also set $c$ to be the critical value $\mathbb{S}(\T\times\{0\})$. The following lemma, which is essentially due to Gromoll and Meyer~\cite[Sect.~2]{gm69}, provides neighborhoods of the critical circle with good properties. 

\begin{lem}\label{l:good_nbhd}
The isolated critical circle $\T\times\{0\}$ has a fundamental system of connected open neighborhoods $\mathcal{U}$ with the following property: There exists $\epsilon=\epsilon(\mathcal{U})>0$ such that $\mathcal{U}\subset\{\mathbb{S}>c-\epsilon\}$ and, if $y\in \mathcal{U}$ and $\phi_{s}(y)\not\in \mathcal{U}$ for some $s>0$, then $\mathbb{S}(\phi_s(y))\leq c-\epsilon$.
\end{lem}

\begin{proof}
We slightly modify Chang's treatment~\cite[page~49]{cha93} in order to deal with our isolated critical circle $\T\times\{0\}$. We can assume without loss of generality that $c=\mathbb{S}(t,0)=0$ for all $t\in\T$. We consider an auxiliary function $\mathbb{G}:\T\times\overline{\mathcal{B}}_R\to\R$ given by
\[
\mathbb{G}(t,z)=\tfrac12 |z|^2 + h\,\mathbb{S}(t,z),
\]
where $h>0$ is a constant that we will determine shortly. The open set $\mathcal{U}$ of the statement will be of the form
\[
\mathcal{U}=\{-\epsilon< \mathbb{S}<\epsilon\}\cap\{\mathbb{G}<g\},
\]
for suitable constants $\epsilon,g>0$. Let us now proceed to determine all the constants. Since $\mathbb{S}$ satisfies the Palais-Smale condition inside $\T\times \overline{\mathcal{B}}_R$, for all $\delta_2>0$, we can find (arbitrarily small) $r\in(0,R)$ and $\delta_1\in(0,\delta_2)$ such that 
\begin{align*}
\delta_2 & \geq |\nabla \mathbb{S}(t,z)|,\qquad\forall  (t,z)\in \T\times \mathcal{B}_r,\\
\delta_1 & \leq |\nabla \mathbb{S}(t,z)|,\qquad\forall  (t,z)\in \T\times (\mathcal{B}_r\setminus \mathcal{B}_{r/2}).
\end{align*}
On the region $\T\times (\mathcal{B}_r\setminus \mathcal{B}_{r/2})$ we have
\begin{align*}
\langle \nabla \mathbb{S},\nabla \mathbb{G}\rangle
= h |\nabla \mathbb{S}|^2 + \langle\nabla \mathbb{S},z\rangle
\geq
h \delta_1^2 - r\delta_2.
\end{align*}
We fix the constant $h > r\delta_2/\delta_1^2$, so that
\begin{align}\label{e:gradients_S_G}
\langle \nabla \mathbb{S}(t,z),\nabla \mathbb{G}(t,z)\rangle > 0,\qquad\forall(t,z)\in\T\times (\mathcal{B}_r\setminus \mathcal{B}_{r/2}).
\end{align}
In order to conclude the proof, it is enough to find values of $\epsilon$ and $g$ such that:
\begin{itemize}
\item[(i)] $\mathcal{U}\subset \T\times \mathcal{B}_r$, 
\item[(ii)] $(\T\times \mathcal{B}_{r/2})\cap\{-\epsilon<\mathbb{S}<\epsilon\}\subset \mathcal{U}$.
\end{itemize}
Let us show that~(i) and~(ii) imply the Lemma. Property~(i) allows us to control the size of $\mathcal{U}$. Property~(ii) implies that $\mathcal{U}$ is a neighborhood of $\T\times\{0\}$ contained in the super-level set $\{\mathbb{S}>-\epsilon\}$. If $\mathcal{U}$ is not connected, we disregard all its connected components other than the one containing $\T\times\{0\}$. Consider a point $y\in \mathcal{U}$ whose forward negative gradient flow orbit is not entirely contained in $\mathcal{U}$, and set
\[s_0:=\min\{s>0\ |\ \phi_s(y)\not\in \mathcal{U}\}.\] 
The point $\phi_{s_0}(y)$ must be contained in the boundary of the open set $\mathcal{U}$, and we have $\mathbb{S}(\phi_{s_0}(y))<\mathbb{S}(y)<\epsilon$. By the definition of $\mathcal{U}$, if $\mathbb{S}(\phi_{s_0}(y))\neq-\epsilon$ we must have $\mathbb{G}(\phi_{s_0}(y))=g$, and by properties~(i--ii) we must have $\phi_{s_0}(y)\in\T\times (\mathcal{B}_r\setminus \mathcal{B}_{r/2})$. By equation~\eqref{e:gradients_S_G}, for all  $s\in(0,s_0)$ sufficiently close to $s_0$, we have $\mathbb{G}(\phi_{s}(y))>g$. This implies that $\phi_{s}(y)\not\in \mathcal{U}$ and contradicts the definition of $s_0$. Therefore we must have $\mathbb{S}(\phi_{s_0}(y))=-\epsilon$, which readily implies the Lemma.

Now, conditions~(i) and~(ii) can be rewritten as
\begin{itemize}
\item[(i)] if $|\mathbb{S}(t,z)|<\epsilon$ and $\mathbb{G}(t,z)<g$, then $|z|<r$;
\item[(ii)] if $|z|<r/2$ and $|\mathbb{S}(t,z)|<\epsilon$, then $\mathbb{G}(t,z)<g$.
\end{itemize}
Condition~(i) is satisfied provided $2g+2h\epsilon<r^2$, while condition~(ii) is satisfied provided $r^2/8+h\epsilon<g$. In order to satisfy them simultaneously, we can choose 
$\epsilon=3 r^2/{32 h}$ and $g={5 r^2}/{16}$.
\end{proof}

In order to state the next result, we go back to the general global setting in which the free-period action functional has the form $\mathbb{S}:\mathcal{M}\to\R$. 

\begin{lem}\label{l:critical_sublevel}
Let $\T\cdot\gamma$ be an isolated critical circle of the free-period action functional $\mathbb{S}$ with critical value $c:=\mathbb{S}(\gamma)$. This circle admits a fundamental system of connected open neighborhoods $\mathcal{U}$ such that $\mathcal{U}\cap\{\mathbb{S}<c\}$ has finitely many connected components. Moreover, if $\ind(\gamma)\geq2$, for every such neighborhood the intersection $\mathcal{U}\cap\{\mathbb{S}<c\}$ is non-empty and connected.
\end{lem}

\begin{proof}
The fundamental system of connected open neighborhoods will be the one given by Lemma~\ref{l:good_nbhd}. Let $\mathcal{U}$ be any open set in this fundamental system, and let $\epsilon=\epsilon(\mathcal{U})>0$ be the associated constant given by Lemma~\ref{l:good_nbhd}. Let $\mathcal{V}$ be another neighborhood of the critical circle $\T\cdot\gamma$ whose closure in contained in $\mathcal{U}\cap\{\mathbb{S}>c-\tfrac\epsilon2\}$. By Lemma~\ref{l:Marino-Prodi}, we can find a function $\mathbb{S}'$ such that $\mathbb{S}'|_{\mathcal{V}}>{c-\tfrac\epsilon2}$, $\mathbb{S'}=\mathbb{S}$ outside $\mathcal{V}$, and $\mathcal{V}$ contains only finitely many  critical points of $\mathbb{S}'$, all of which are non-degenerate and with Morse index larger than or equal to $\ind(\gamma)$. Since $\mathcal{U}\cap\{\mathbb{S}\leq c-\tfrac\epsilon2\}$ is the exit set of $\mathcal{U}$ for the negative gradient flow of $\mathbb{S}'$ in forward time, we have the Morse inequality
\begin{align}\label{e:Morse_inequality}
\rank H_1(\mathcal{U},\mathcal{U}\cap\{\mathbb{S}\leq c-\tfrac\epsilon2\})
\leq 
\# \crit_1 (\mathbb{S}'|_{\mathcal{V}}) <\infty,
\end{align}
where $\crit_1 (\mathbb{S}'|_{\mathcal{V}})$ denotes the set of critical points of $\mathbb{S}'|_{\mathcal{V}}$ with Morse index 1. Let $r$ be the number of path-connected components of the subspace $\mathcal{U}\cap\{\mathbb{S}\leq c-\tfrac\epsilon2\}$. The Morse inequality \eqref{e:Morse_inequality}, together with the  homology long exact sequence
\[
...\longrightarrow H_1(\mathcal{U},\mathcal{U}\cap\{\mathbb{S}\leq c-\tfrac\epsilon2\};\Z)\longrightarrow H_0(\mathcal{U}\cap\{\mathbb{S}\leq c-\tfrac\epsilon2\};\Z)\longrightarrow H_0(\mathcal{U};\Z)\longrightarrow ...
\]
implies  
\begin{align}\label{e:bound_on_connected_components}
r=\rank H_0(\mathcal{U}\cap\{\mathbb{S}\leq c-\tfrac\epsilon2\};\Z)\leq \# \crit_1 (\mathbb{S}'|_{\mathcal{V}}) +1<\infty.
\end{align}
Now, we denote by $\theta_t$ the flow of the renormalized gradient $-\nabla \mathbb{S}/|\nabla \mathbb{S}|^2$, which is a well defined  smooth vector field on $\mathcal{U}\cap\{\mathbb{S}<c\}$. Notice that, if $\theta_s(\beta)$ is well-defined for all $s\in[0,t]$, we have $\mathbb{S}(\beta) - \mathbb{S}(\theta_t(\beta)) = t$. Let $\tau:\mathcal{U}\cap\{\mathbb{S}<c\}\to[0,\infty)$ be the continuous function given by $\tau(\beta)=\max\{0,\mathbb{S}(\beta)-c+\epsilon/2\}$. By the properties of $\mathcal{U}$, we have
\[
\theta_{\tau(\beta)}(\beta)\in \mathcal{U}\cap\{\mathbb{S}\leq c-\tfrac\epsilon2\},
\qquad
\forall \beta\in \mathcal{U}\cap\{\mathbb{S}<c\}.
\]
This shows that  $\mathcal{U}\cap\{\mathbb{S}\leq c-\tfrac\epsilon2\}$ is a deformation retract of $\mathcal{U}\cap\{\mathbb{S}<c\}$, and therefore this latter open set  has $r$ path-connected components.

As for the ``moreover'' part of the lemma, assume that $\ind(\gamma)\geq2$. In particular $\gamma$ is not a local minimum, and therefore $r\geq1$. But since all the critical points of the function $\mathbb{S}'|_{\mathcal{V}}$ have Morse index larger than or equal to $\ind(\gamma)$, the inequality~\eqref{e:bound_on_connected_components} implies that $r\leq1$, and we conclude that $r=1$.
\end{proof}

\subsection{Iterated mountain passes}
After these preliminaries, we can finally prove the main result of this section, which shows that critical circles of the free-period action functional cannot be of mountain pass type if iterated sufficiently many times.

\begin{thm}\label{t:iterated_mountain_pass}
Let $\T\cdot \gamma $ be a critical circle of the free-period action functional $\mathbb{S}$ with critical value $c=\mathbb{S}(\gamma)$. Assume that all the iterates of $\gamma$ belong to  isolated critical circles of $\mathbb{S}$. Then, for all integers $n$ large enough there exists an open neighborhood $\mathcal{W}$ of   $\T\cdot\gamma^n$ with the following property:  If any two points $\gamma_0,\gamma_1\in\{\mathbb{S}<nc\}$ are contained in the  same connected component of  $\{\mathbb{S}<nc\}\cup \mathcal{W}$, they are actually contained in the same connected component of  $\{\mathbb{S}<nc\}$.
\end{thm}

\begin{proof}
Assume first that the mean Morse index $\overline{\mathrm{ind}}(\gamma)$ is positive. In this case, for $n$ large enough the Morse index $\ind(\gamma^n)$ is larger than one. By Lemma~\ref{l:critical_sublevel}, the critical circle $\T\cdot\gamma^n$ has an open neighborhood $\mathcal{W}$ whose intersection with the sublevel set $\{\mathbb{S}<nc\}$ is connected, and our theorem readily follows. 

Now,  assume that the mean Morse index $\overline{\mathrm{ind}}(\gamma)$ is zero. By Proposition~\ref{p:iterated_Morse_index}, this is equivalent to the fact that the Morse index of $\gamma^n$  is zero for all $n\in\N$. Lemma~\ref{l:iterated_nullity} further implies that there exists a partition $\N=\N_1\cup...\cup\N_k$, integers $n_1\in\N_1$, ..., $n_k\in\N_k$, and $\nu_1,...,\nu_k\in\{0,...,2d-2\}$ with the following property: The integers in $\N_i$ are all multiples of $n_i$, and for all $n\in\N_i$ the critical point $\gamma^n$ has nullity $\nu_i$. 

Let us choose, once and for all, $i\in\{1,...,k\}$ and consider integers $n$ belonging to $\N_i$. Since $\T\cdot\gamma^{n_i}$ is an isolated critical circle of $\mathbb{S}$, by Lemma~\ref{l:critical_sublevel} it has a connected open neighborhood $\mathcal{U}$ such that the open subset $\mathcal{U}^-:=\mathcal{U}\cap\{\mathbb{S}<n_ic\}$ has finitely many connected components $\mathcal{U}_1^-,...,\mathcal{U}_r^-$.  For each $\alpha\in\{1,...,r\}$ we fix an arbitrary point $\gamma_\alpha\in \mathcal{U}_\alpha^-$, and for each pair of distinct $\alpha,\beta\in\{1,...,r\}$ we fix a continuous path
$\Theta_{\alpha\beta}:[-1,1]\to \mathcal{U}$
joining $\gamma_\alpha$ and $\gamma_\beta$. By Bangert's technique of ``pulling one loop at a time'' (see \cite[pages~86--87]{ban80} or \cite[page~421]{abb13}), for all sufficiently large multiples $n$ of $n_i$, each iterated path $\psi^{n/n_i}\circ\Theta_{\alpha\beta}$ is homotopic with fixed endpoints to a path entirely contained in the sublevel set $\{\mathbb{S}<nc\}$. In other words, $\psi^{n/n_i}(\mathcal{U}^-)$ is contained in a path-connected component of $\{\mathbb{S}<nc\}$ provided $n\in\N_i$ is larger than some number $\overline{n}_i$.

Let us consider an integer $n\in\N_i$ larger than the constant $\overline{n}_i$. By Lemma~\ref{l:Gromoll-Meyer}, if we choose a sufficiently small neighborhood $\mathcal{U}'\subset \mathcal{U}$ of the critical circle $\T\cdot\psi^{n_i}(\gamma)$, its image $\psi^{n/n_i}(\mathcal{U}')$ has a tubular neighborhood $\mathcal{V}$ with an associated deformation retraction $r_t:\mathcal{V}\to \mathcal{V}$ such that $r_0=\mathrm{id}$, $r_1(\mathcal{V})=\psi^{n/n_i}(\mathcal{U}')$, and $\tfrac{d }{d  t} \mathbb{S}\circ r_t\leq0$. We set $\mathcal{W}$ to be an open neighborhood of $\T\cdot\gamma^n$ that is small enough so that its closure is contained in $\mathcal{V}$. Consider two points $\gamma_0, \gamma_1\in\{\mathbb{S}<nc\}$ as in the statement, and  a continuous path $\Gamma:[0,1]\to\{\mathbb{S}<nc\}\cup \mathcal{W}$ joining them. We denote by $s_0$ and $s_1$ respectively the infimum and the supremum of all $s\in[0,1]$ such that $\Gamma(s)\in \mathcal{W}$. The points $\Gamma(s_0)$ and $\Gamma(s_1)$ lie in the intersection $\mathcal{V}\cap\{\mathbb{S}<nc\}$. Since the  deformation retraction $r_t$ does not increase the free-period action functional $\mathbb{S}$, the points $\Gamma(s_0)$ and $r_1(\Gamma(s_0))$ are contained in the same connected component of the sublevel set $\{\mathbb{S}<nc\}$. Moreover, $r_1(\Gamma(s_0))$ is contained in $\psi^{n/n_i}(\mathcal{U}^-)$. Analogously, $\Gamma(s_1)$ and $r_1(\Gamma(s_1))$ are contained in the same connected component of $\{\mathbb{S}<nc\}$, and $r_1(\Gamma(s_1))$ lies in $\psi^{n/n_i}(\mathcal{U}^-)$. Since, as we proved in the previous paragraph of the proof, $\psi^{n/n_i}(\mathcal{U}^-)$ is contained in a connected component of  $\{\mathbb{S}<nc\}$, we conclude that $\gamma_0$ and $\gamma_1$ belong to the same connected component of $\{\mathbb{S}<nc\}$.
\end{proof}

Theorem~\ref{t:iterated_mountain_pass} should be compared to \cite[Theorem 2]{ban80}, which gives a similar statement for isolated closed Riemannian geodesics on surfaces and is proved using geometric arguments. A rather immediate corollary  is the following generalization of the waist theorem of Bangert.

\begin{cor}\label{c:waist}
Assume that $M$ is an orientable surface, that the free-period action functional $\mathbb{S}$ satisfies the Palais-Smale condition, and that it admits a local minimum $\gamma$ with action $\mathbb{S}(\gamma)\neq0$ whose critical circle $\T\cdot\gamma$ is not the whole set of global minima of $\mathbb{S}$ in its connected component. Then, there are infinitely many periodic orbits with energy $0$.
\end{cor}

The proof of this corollary is a minor variation of the argument in Section~\ref{s:proof}, and we leave its details to the reader. Notice that, if the Ma\~n\'e critical value of the universal cover $c_u$ is negative, the free-period action functional $\mathbb{S}$ satisfies the Palais-Smale condition (see  \cite[Lemmata~5.1 and~5.4]{abb13}). Moreover, in this situation, any contractible local minimum of $\mathbb{S}$ has positive action, and  it is never a global minimum. In particular, the existence of one such minimum is enough to infer the existence of infinitely many other periodic orbits on the energy hypersurface. 

\begin{rem}
\label{waistrem}
We expect Corollary~\ref{c:waist} to hold  for more general Lagrangians $L$, such as the square of a Finsler norm. In the latter case, the  functional $\mathbb{S}$ is not twice differentiable, but our proof makes a minimal use of higher regularity and it should be possible to  adapt it to this situation. As already mentioned in the introduction, this would permit to generalize the waist theorem of Bangert to Finsler metrics on $S^2$. We also expect Corollary~\ref{c:waist} to hold for some non-exact magnetic flows, for instance in case $M$ is a surface of higher genus, for which  a suitable free-time action functional is still available.
\end{rem}

\section{The minimax argument}\label{minmaxarg}

\subsection{The sequence of minimax functions}\label{minimaxsec}

Throughout Section~\ref{minmaxarg} we will assume our closed surface $M$ to be orientable. This is not a restrictive assumption for us: If $M$ is non-orientable, we can replace it by its orientation double cover and work there. In fact, the existence of infinitely many periodic orbits for the orientation double cover of $M$ clearly implies the same result for $M$.

We take the energy $\kappa$ into account, by considering the one-parameter family of functionals $\mathbb{S}_{\kappa} : \mathcal{M}\rightarrow \R$ given by
\[
\mathbb{S}_{\kappa}(x,T) := T \int_{\T} \bigl( L(x(s),\dot{x}(s)/T) + \kappa \bigr)\, ds.
\]
Critical points of $\mathbb{S}_{\kappa}$ are in one-to-one correspondence with periodic orbits of energy $\kappa$: More precisely, a critical point $(x,T)$ is associated to the $T$-periodic orbit $\gamma(t)=x(t/T)$ which has energy $E(\gamma,\dot\gamma)=\kappa$. Notice that $T$ is not necessarily the minimal period of $\gamma$.

In the remaining part of the paper we endow $\mathcal{M} = H^1(\T,M) \times (0,+\infty)$ with the standard product metric as in Section~\ref{s:basic_properties} and with the induced distance. 

For every $\kappa\in (0,c_0)$ the functional $\mathbb{S}_{\kappa}$ has a local minimizer $\alpha_{\kappa}$ with
\[
\mathbb{S}_{\kappa}(\alpha_{\kappa}) <0.
\] 
The proof of this fact is contained in  \cite[Lemma 3.2]{amp13} and  builds on previous results by Taimanov \cite{tai92b, tai92c,tai92} and Contreras, Macarini and Paternain \cite{cmp04}. Since $M$ is an orientable surface, every iteration $\alpha_{\kappa}^n$ of $\alpha_{\kappa}$ remains a local minimizer of $\mathbb{S}_{\kappa}$, see \cite[Lemma 4.1]{amp13}. Moreover, if the local minimizer $\alpha_{\kappa}$ is strict, meaning that
\[
\mathbb{S}_{\kappa}(\gamma) > \mathbb{S}_{\kappa}(\alpha_{\kappa}), \qquad \forall \gamma \in \mathcal{U} \setminus \bigl( \T\cdot \alpha_{\kappa} \bigr),
\]
for some neighborhood $\mathcal{U}$ of the critical circle $\T\cdot \alpha_{\kappa}$, so are all the iterates $\alpha_{\kappa}^n$.

Fix some $\kappa_*$ in the interval $(0,c_u)$ such that $\alpha_{\kappa_*}$ is a strict local minimizer of $\mathbb{S}_{\kappa_*}$. Since $\kappa_*$ is strictly smaller than $c_u$, the infimum of $\mathbb{S}_{\kappa_*}$ over all contractible closed curves is $-\infty$, and hence we can find an element $\mu\in \mathcal{M}$ in the same free homotopy class of $\alpha_{\kappa_*}$ such that 
\[
\mathbb{S}_{\kappa_*} (\mu) < \mathbb{S}_{\kappa_*}(\alpha_{\kappa_*}).
\]
Choose a bounded open neighborhood $\mathcal{U}$ of $\T\cdot \alpha_{\kappa_*}$ whose closure intersects the critical set of $\mathbb{S}_{\kappa_*}$ only in $\T\cdot \alpha_{\kappa_*}$, such that
\[
\inf_{\partial \mathcal{U}} \mathbb{S}_{\kappa_*} > \mathbb{S}_{\kappa_*} (\alpha_{\kappa_*}),
\]
and such that $T$ is bounded away from zero for $(x,T)\in \mathcal{U}$.
The existence of such a neighborhood is an easy consequence of the fact that $\mathbb{S}_{\kappa_*}$ satisfies the Palais-Smale condition on bounded subsets on which $T$ is bounded away from zero, see \cite[Lemma 5.3]{amp13}. We denote by $M_{\kappa}$ the closure of the set of local minimizers of $\mathbb{S}_{\kappa}$ which belong to $\mathcal{U}$. Such a set consists of critical points of $\mathbb{S}_{\kappa}$, but in general may contain critical points which are not local minimizers. If $M_{\kappa}$ is a a finite union of critical circles, then all its elements are strict local minimizers.

\begin{lem} 
\label{LJ}
There exists a closed interval $J=J(\kappa_*)\subset (0,c_u)$ whose interior part contains $\kappa_*$ and which has the following properties:
\begin{enumerate}
\item For every $\kappa\in J$ the set $M_{\kappa}$ is a non-empty compact set.
\item For every $\kappa\in J$ there holds
\[
\mathbb{S}_{\kappa}(\mu) < \min_{M_{\kappa}} \mathbb{S}_{\kappa}.
\]
\item For every $\kappa\in J$ there holds
\[
\sup_{\kappa'\in J } \max_{M_{\kappa'}} \mathbb{S}_{\kappa} < \min\bigl\{ \inf_{\partial \mathcal{U}} \mathbb{S}_{\kappa}, 0\bigr\}.
\]
\end{enumerate}
\end{lem}

\begin{proof}
This lemma is an easy consequence of the fact that on bounded subsets of $\mathcal{M}$ the family of functionals $\mathbb{S}_{\kappa}$ converges to $\mathbb{S}_{\kappa_*}$ in the $C^1$ norm  for $\kappa\rightarrow \kappa_*$, because
\[
\mathbb{S}_{\kappa}(x,T) - \mathbb{S}_{\kappa_*}(x,T) = (\kappa-\kappa_*)T,
\]
and of the fact that $\mathbb{S}_{\kappa}$ satisfies the Palais-Smale condition on $\mathcal{U}$. Indeed, fix numbers $A_0,A_1,A_2,A_3 $ and $A_4$ such that
\[
\mathbb{S}_{\kappa_*}(\mu) < A_0 < A_1 < \mathbb{S}_{\kappa_*}(\alpha_{\kappa_*}) < A_2 < A_3 < A_4 < A_5:= \min\bigl\{ \inf_{\partial \mathcal{U}} \mathbb{S}_{\kappa_*},0\bigr\}.
\]
Let $\mathcal{U}'\subset \mathcal{U}$ be a neighborhood of $\T\cdot \alpha_{\kappa_*}$ such that $\mathbb{S}_{\kappa_*}(\mathcal{U}') \subset (A_1,A_2)$. Since the closure of the bounded set $\mathcal{U}\setminus \mathcal{U}'$ does not contain critical points of $\mathbb{S}_{\kappa_*}$, by the Palais-Smale condition there exists a positive number $\delta$ such that
\[
\|d\mathbb{S}_{\kappa_*}\| \geq \delta \qquad \mbox{on} \quad \mathcal{U}\setminus \mathcal{U}'.
\]
By the $C^1$ convergence of $\mathbb{S}_{\kappa}|_{\mathcal{U}}$ to $\mathbb{S}_{\kappa_*}|_{\mathcal{U}}$, we can find a neighborhood $J_0\subset (0,c_u)$ of $\kappa_*$ such that
\[
\|d\mathbb{S}_{\kappa}\| \geq \delta/2 \qquad \mbox{on} \quad \mathcal{U}\setminus \mathcal{U}', \; \forall \kappa\in J_0.
\]
In particular, the critical set of $\mathbb{S}_{\kappa}|_{\mathcal{U}}$ is contained in $\mathcal{U}'$ for every $\kappa\in J_0$, and so is its subset $M_{\kappa}$. Since $\mathbb{S}_{\kappa_*}(\mu) < A_0$, $\mathbb{S}_{\kappa_*}(\mathcal{U}')\subset (A_1,A_2)$ and $\mathbb{S}_{\kappa_*}(\partial \mathcal{U})\subset [A_5,+\infty)$, we can find a closed interval $J\subset J_0$ which is a neighborhood of $\kappa_*$ such that
\begin{equation}
\label{livelli}
\mathbb{S}_{\kappa}(\mu) < A_0, \qquad
\mathbb{S}_{\kappa}(\mathcal{U}') \subset (A_0,A_3), \qquad \mathbb{S}_{\kappa}(\partial \mathcal{U}) \subset (A_4,+\infty), \qquad \forall \kappa \in J.
\end{equation}
The interval $J$ satisfies the required properties. Indeed, by (\ref{livelli}) the infimum of $\mathbb{S}_{\kappa}$ on $\mathcal{U}$ is strictly smaller than its infimum on $\partial \mathcal{U}$ and hence, by the Palais-Smale condition, $\mathbb{S}_{\kappa}|_{\mathcal{U}}$ has a global minimizer. Therefore, $M_{\kappa}$ is not empty and, again by the Palais-Smale condition, compact. So (i) holds. Since $M_{\kappa}$ is contained in $\mathcal{U}'$, the minimum of $\mathbb{S}_{\kappa}$ on $M_{\kappa}$ is greater than $A_0$, which is greater than $S_{\kappa}(\mu)$, by (\ref{livelli}), proving (ii). Finally, (\ref{livelli}) implies that
\[
\sup_{\kappa'\in J } \max_{M_{\kappa'}} \mathbb{S}_{\kappa} \leq \sup_{\mathcal{U}'} \mathbb{S}_{\kappa} \leq A_3 < A_4 \leq \min\left\{ \inf_{\partial \mathcal{U}} \mathbb{S}_{\kappa},0\right\}, \qquad \forall \kappa\in J,
\]
which proves (iii).
\end{proof}

Property (iii) of the above lemma is used in the following:

\begin{lem}
\label{montecarlo}
Let $\kappa_0<\kappa_1$ be in $J$. For every $\alpha \in M_{\kappa_1}$ there exists a continuous path $w:[0,1]\rightarrow \mathcal{U}$ such that $w(0)\in M_{\kappa_0}$, $w(1)=\alpha$ and $\mathbb{S}_{\kappa_0}\circ w \leq \mathbb{S}_{\kappa_0}(\alpha)$.
\end{lem}

\begin{proof}
The element $\alpha$ corresponds to a periodic orbit of energy $\kappa_1$ and in particular is not a critical point of $\mathbb{S}_{\kappa_0}$. Set $a:=\mathbb{S}_{\kappa_0}(\alpha)$. Being a regular point of the hypersurface $\mathbb{S}_{\kappa_0}^{-1}(a)$, $\alpha$ can be connected to a point $\beta \in \mathcal{U}\cap \{\mathbb{S}_{\kappa_0}< a\}$ by a continuous path which is contained in the sublevel set $\{\mathbb{S}_{\kappa_0}\leq a\}$. By Lemma \ref{LJ} (iii),
\[
a=\mathbb{S}_{\kappa_0}(\alpha) < \inf_{\partial \mathcal{U}} \mathbb{S}_{\kappa_0},
\]
and hence the connected component of $\{\mathbb{S}_{\kappa_0}\leq a\}$ which contains $\alpha$ is contained in $\mathcal{U}$. Since $\mathbb{S}_{\kappa_0}$ satisfies the Palais-Smale condition on $\mathcal{U}$, the above fact ensures the existence of a global minimizer $\gamma$ of the restriction of $\mathbb{S}_{\kappa_0}$ to the connected component of $\{\mathbb{S}_{\kappa_0}< a\}$ which contains $\beta$. Such a $\gamma$ belongs to $M_{\kappa_0}$ and can be connected to $\beta$ by a continuous path in $\{\mathbb{S}_{\kappa_0}< a\}$. We conclude that there exists a continuous path
$w:[0,1] \rightarrow \{\mathbb{S}_{\kappa_0}\leq a\}$ 
such that $w(0)=\gamma\in M_{\kappa_0}$ and $w(1)=\alpha$.
\end{proof}

 For every $n\in \N$ and every $\kappa\in J$ we define the set of continuous paths
 \[
 \mathcal{P}_n(\kappa) := \set{u\in C^0([0,1], \mathcal{M})}{u(0)\in \psi^n(M_{\kappa}), \; u(1)=\mu^n},
 \]
 which join the $n$-th iterate of some element in $M_{\kappa}$ with the $n$-th iterate of $\mu$. Correspondingly, we define the minimax value
\[
c_n(\kappa) := \inf_{u\in \mathcal{P}_n(\kappa)} \max_{\sigma\in [0,1]} \mathbb{S}_{\kappa}(u(\sigma)).
\]
By Lemma \ref{LJ} (ii) there holds
\begin{equation}
\label{prima}
c_n(\kappa) \geq \min_{\psi^n(M_{\kappa})} \mathbb{S}_{\kappa} = n \min_{M_{\kappa}} \mathbb{S}_{\kappa} > n \, \mathbb{S}_{\kappa}(\mu).
\end{equation}

\begin{lem}
For every $n\in \N$ the function $c_n$ is monotonically increasing on $J$.
\end{lem}
 
\begin{proof}
Let $\kappa_0<\kappa_1$ be numbers in $J$. Let $u\in \mathcal{P}_n(\kappa_1)$. Then
$u(0)$ is the $n$-th iterate of an element of $M_{\kappa_1}$ and by Lemma \ref{montecarlo} we can join $u(0)$ to the $n$-th iterate of some element of $M_{\kappa_0}$ by a path in $\{\mathbb{S}_{\kappa_0}\leq \mathbb{S}_{\kappa_0}(u(0))\}$.
By concatenation we obtain a path
$v:[0,1] \rightarrow \mathcal{M}$
such that $v(0)\in \psi^n(M_{\kappa_0})$, $v([0,1/2])\subset \{\mathbb{S}_{\kappa_0}\leq \mathbb{S}_{\kappa_0}(u(0))\}$, and $v|_{[1/2,1]} = u(2(\cdot-1/2))$. Then $v$ belongs to $\mathcal{P}_n(\kappa_0)$ and, since $\mathbb{S}_{\kappa_0}\leq \mathbb{S}_{\kappa_1}$, we have
\[
\max_{[0,1]} \mathbb{S}_{\kappa_0}\circ v \leq \max_{[0,1]} \mathbb{S}_{\kappa_0}\circ u \leq \max_{[0,1]} \mathbb{S}_{\kappa_1}\circ u.
\]
Therefore,
\[
c_n(\kappa_0) \leq \max_{[0,1]} \mathbb{S}_{\kappa_0}\circ v \leq \max_{[0,1]} \mathbb{S}_{\kappa_1}\circ u,
\]
and by taking the infimum over all $u\in \mathcal{P}_n(\kappa_1)$ we conclude that
$c_n(\kappa_0) \leq c_n(\kappa_1)$.
\end{proof}

The next lemma is a simple generalization of \cite[Lemma 6.2]{amp13}, where we replace the elements $\mu_0,\mu_1\in \mathcal{M}$ by compact subsets $K_0,K_1\subset \mathcal{M}$. Its proof is based on Bangert's argument from \cite{ban80}.

\begin{lem}
Let $K_0$ and $K_1$ be compact sets in the same connected component of $\mathcal{M}$ and set
\[
\mathcal{R}_n := \set{u\in C^0([0,1],\mathcal{M})}{u(0)\in \psi^n(K_0), \; u(1)\in \psi^n(K_1)}.
\]
For some fixed $\kappa$ we set
\[
a_n := \inf_{u\in \mathcal{R}_n} \max_{\sigma\in [0,1]} \mathbb{S}_{\kappa}(u(\sigma)).
\]
Then there exists a number $A$ such that
\[
a_n \leq n \max_{K_0 \cup K_1} \mathbb{S}_{\kappa} + A \qquad \forall n\in \N.
\]
\end{lem}

If we apply the above lemma to $\hat{\kappa}=\max J$, $K_0=M_{\hat{\kappa}}$, $K_1=\{\mu\}$ and use the fact that $\mathbb{S}_{\hat{\kappa}}(\mu)$ is negative and, by Lemma \ref{LJ} (iii), $\max_{M_{\hat{\kappa}}} \mathbb{S}_{\hat{\kappa}}$ is also negative, we obtain in particular that $c_n(\hat\kappa)\rightarrow -\infty$ for $n\rightarrow \infty$.
Since $c_n$ is monotonically increasing, we conclude that
\begin{equation}
\label{neg}
\lim_{n\rightarrow \infty} c_n = -\infty \qquad \mbox{uniformly on } J.
\end{equation}

\subsection{The monotonicity argument}

Let $c_n:J \rightarrow \R$ be the sequence of minimax functions which is defined in the previous section. By (\ref{neg}) there exists a natural number $n_0$ such that all the $c_n$'s are negative for $n\geq n_0$. The proof of the next lemma is based on Struwe's monotonicity argument from \cite{str90} and is similar to the proof of \cite[Lemma 7.1]{amp13}. The fact that we are dealing with a peculiar minimax class requires some extra care, and therefore we include a full proof.

\begin{lem}
\label{struwe}
Let $n\geq n_0$. Let $\bar{\kappa}$ be an interior point of $J$ at which $c_n$ is differentiable and such that the set $M_{\bar\kappa}$ is a finite union of critical circles. Then for every neighborhood $\mathcal{V}$ of the set 
\[
\mathrm{crit} (\mathbb{S}_{\bar\kappa}) \cap \mathbb{S}_{\bar\kappa}^{-1}(c_n(\bar\kappa))
\]
there exists an element $u$ of $\mathcal{P}_n(\bar\kappa)$ such that $\mathbb{S}_{\bar\kappa}(u(0)) < c_n(\bar\kappa)$ and
\[
u([0,1]) \subset \{ \mathbb{S}_{\bar\kappa} < c_n(\bar\kappa)\} \cup \mathcal{V}.
\]
In particular, $c_n(\bar\kappa)$ is a critical value of $\mathbb{S}_{\bar\kappa}$.
\end{lem}

\begin{proof}
The last assertion follows from the previous one by arguing by contradiction and choosing $\mathcal{V}$ to be the empty set. Therefore, we must prove the existence of a path $u$ which satisfies the above requirements.

Since $c_n$ is differentiable at the interior point $\bar\kappa\in J$, there exist a neighborhood $I$ of $\bar\kappa$ which is contained in $J$ and a number $C\geq 0$ such that
\begin{equation}
\label{linmod}
|c_n(\kappa)-c_n(\bar\kappa)| \leq C |\kappa - \bar\kappa|, \qquad \forall \kappa\in I.
\end{equation}
Let $(\kappa_h)\subset I$ be a strictly decreasing sequence which converges to $\bar\kappa$, and consider the infinitesimal sequence of positive numbers $\epsilon_h := \kappa_h - \bar\kappa$. Let $u_h$ be an element of $\mathcal{P}_n(\kappa_h)$ such that
\[
\max_{[0,1]} \mathbb{S}_{\kappa_h} \circ u_h \leq c_n(\kappa_h) + \epsilon_h.
\]
Let $\gamma=(x,T)$ be an element in the image of $u_h$. From (\ref{linmod}) we deduce that
\begin{equation}
\label{b1}
\mathbb{S}_{\bar\kappa}(\gamma) \leq \mathbb{S}_{\kappa_h}(\gamma) \leq c_n(\kappa_h) + \epsilon_h \leq c_n(\bar\kappa) + (C+1)\epsilon_h.
\end{equation}
If moreover $\gamma\in u_h([0,1])$ is such that $\mathbb{S}_{\bar\kappa}(\gamma)> c_n(\bar\kappa)-\epsilon_h$, then a second application of (\ref{linmod}) gives us the bound
\[
T = \frac{\mathbb{S}_{\kappa_h}(x,T) - \mathbb{S}_{\bar\kappa}(x,T)}{\kappa_h - \bar\kappa} \leq \frac{c_n(\kappa_h) + \epsilon_h - c_n(\bar\kappa)+\epsilon_h}{\epsilon_h} \leq C + 2.
\] 
It follows that
\[
u_h([0,1]) \subset \{ \mathbb{S}_{\bar\kappa} \leq c_n(\bar\kappa)-\epsilon_h\} \cup \mathcal{A}_h,
\]
where 
\[
\mathcal{A}_h := \set{(x,T)\in \mathcal{M}}{T \leq C+2, \; \mathbb{S}_{\bar\kappa}(x,T) \leq c_n(\bar\kappa) + (C+1)\epsilon_h}.
\]
By the estimate
\[
\mathbb{S}_{\bar\kappa}(x,T) = \frac{1}{2T} \int_{\T} |\dot{x}(s)|^2_{x(s)} \, ds - \int_{\T} x^*(\theta) + \bar\kappa T \geq \frac{1}{2T} \|\dot{x}\|_{L^2}^2 - \|\theta\|_{\infty} \|\dot{x}\|_{L^2},
\]  
the set $\mathcal{A}_h$ is bounded in $\mathcal{M}$, uniformly with respect to $h\in \N$. 

The point $u_h(0)$ belongs to $\psi^n(M_{\kappa_h})$ and by Lemma \ref{montecarlo} it can be joined to some element in $\psi^n(M_{\bar\kappa})$ by a path which remains in $\psi^n(\mathcal{U})$ and in 
\[
\{\mathbb{S}_{\bar\kappa} \leq \mathbb{S}_{\bar\kappa}(u_h(0))\} \subset \{\mathbb{S}_{\bar\kappa} \leq c_n(\bar\kappa) + (C+1) \epsilon_h\},
\]
where we have used again (\ref{b1}). By concatenating the latter path with $u_h$, we obtain a path $v_h : [0,1] \rightarrow \mathcal{M}$ such that $v_h(0)\in \psi^n(M_{\bar\kappa})$, $v_h(1)=\mu^n$, and
\[
v_h([0,1]) \subset \{ \mathbb{S}_{\bar\kappa} \leq c_n(\bar\kappa)-\epsilon_h\} \cup \mathcal{A}_h \cup \bigl( \psi^n(\mathcal{U}) \cap \{\mathbb{S}_{\bar\kappa} \leq c_n(\bar\kappa)+ (C+1)\epsilon_h\} \bigr).
\]
In particular, $v_h$ belongs to $\mathcal{P}_n(\bar\kappa)$.
From the uniform boundedness of $\mathcal{A}_h$ and the fact that $\psi^n(\mathcal{U})$ is also bounded, we deduce that there exists a bounded set $\mathcal{B}\subset \mathcal{M}$ which is independent of $h\in \N$ and such that
\begin{equation}
\label{dovev_h}
v_h([0,1]) \subset \{ \mathbb{S}_{\bar\kappa} \leq c_n(\bar\kappa)-\epsilon_h\} \cup \bigl( \mathcal{B} \cap \{\mathbb{S}_{\bar\kappa} \leq c_n(\bar\kappa)+ (C+1)\epsilon_h\} \bigr).
\end{equation}
In particular, we have
\begin{equation}
\label{dovec}
\limsup_{h\rightarrow \infty} \max_{[0,1]} \mathbb{S}_{\bar\kappa}\circ v_h \leq c_n(\bar\kappa).
\end{equation}
Since by assumption $M_{\bar\kappa}$ is a finite union of critical circles, it consists of strict local minimizers for $\mathbb{S}_{\bar\kappa}$. By the already mentioned \cite[Lemma 4.1]{amp13}, also $\psi^n(M_{\bar\kappa})$ consists of strict local minimizers for $\mathbb{S}_{\bar\kappa}$. Since $\psi^n(M_{\bar\kappa})$ consists of finitely many critical circles, up to the choice of a subsequence we may assume that $v_h(0)$ belongs to the same critical circle $\T\cdot \alpha$ for every $h\in \N$. In particular, 
\[
A_0:= \mathbb{S}_{\bar\kappa}(v_h(0))=\mathbb{S}_{\bar\kappa}(\alpha)
\] 
does not depend on $h$. Since $\T\cdot \alpha$ minimizes $\mathbb{S}_{\bar\kappa}$ strictly, it has a neighborhood $\mathcal{W}$ on which $\mathbb{S}_{\bar\kappa}\geq A_0$ and such that
\[
\inf_{\partial \mathcal{W}} \mathbb{S}_{\bar\kappa} > A_0.
\]
See the already mentioned \cite[Lemma 5.3]{amp13}.
Since 
\[
\mathbb{S}_{\bar\kappa}(v_h(1)) = \mathbb{S}_{\bar\kappa}(\mu^n) < \mathbb{S}_{\bar\kappa}(v_h(0))
\]
because of Lemma \ref{LJ} (ii), every path $v_h$ must meet the boundary of $\mathcal{W}$ and hence
\[
\inf_{h\in \N} \max_{[0,1]} \mathbb{S}_{\bar\kappa}\circ v_h > A_0.
\]
Together with (\ref{dovec}), the above strict inequality implies that
\[
c_n(\bar\kappa)> A_0.
\]
Using also the fact that $c_n(\bar\kappa)$ is negative (because $n\geq n_0$), we can find numbers $A_1$ and $A_2$ such that
\[
A_0 < A_1 < c_n(\bar\kappa) < A_2 < 0.
\]

Let $\nabla \mathbb{S}_{\bar\kappa}$ be the gradient vector field of $\mathbb{S}_{\bar\kappa}$ with respect to the standard product Riemannian metric on $\mathcal{M} = H^1(\T,M) \times (0,+\infty)$. By multiplying $-\nabla \mathbb{S}_{\bar\kappa}$ by a suitable smooth non-negative function we can construct a smooth bounded tangent vector field $X$ on $\mathcal{M}$ which vanishes on
\[
\{\mathbb{S}_{\bar\kappa}\leq A_0\} \cup \{\mathbb{S}_{\bar\kappa}\geq 0\},
\]
satisfies $d\mathbb{S}_{\bar\kappa} (X)\leq 0$ on $\mathcal{M}$, and 
\begin{equation}
\label{decr}
d\mathbb{S}_{\bar\kappa} (X) \leq - \min\{ \|d\mathbb{S}_{\bar\kappa}\|^2,1\} \qquad \mbox{on} \quad \{A_1 \leq \mathbb{S}_{\bar\kappa} \leq A_2\}.
\end{equation}
The flow $\phi$ of $X$ is well-defined on $\mathcal{M}\times [0,+\infty)$, because $X$ is bounded and the only source of non-completeness is $T$ going to zero, but this may happen only for negative-gradient flow lines for which $\mathbb{S}_{\bar\kappa}$ tends to zero (see \cite[Lemmata 3.2 and 3.3]{abb13}), and we have made $X$ vanish on $\{\mathbb{S}_{\bar\kappa}\geq 0\}$. Furthermore, since $X$ is bounded $\phi$ maps bounded sets into bounded sets.

We claim that if $h$ is large enough, then
\[
\phi_1(v_h([0,1])) \subset \{\mathbb{S}_{\bar\kappa}<c_n(\bar\kappa)\} \cup \mathcal{V}.
\]
This claim implies the thesis of this lemma: Indeed, the path $\phi_1 \circ v_h$ belongs to $\mathcal{P}_n(\bar\kappa)$, because 
\[
\mathbb{S}_{\bar\kappa}(\phi_1 \circ v_h(0)) = A_0 < c_n(\bar\kappa), \qquad \mathbb{S}_{\bar\kappa}(\phi_1 \circ v_h(1)) = \mathbb{S}_{\bar\kappa}(\mu^n) < A_0, 
\]
and $\phi$ fixes the points in $\{\mathbb{S}_{\bar\kappa}\leq A_0\}$.

There remains to prove the above claim. By (\ref{dovev_h}) and the properties of $X$,
\begin{equation}
\label{dovephi}
\phi([0,1] \times v_h([0,1]) \subset \{ \mathbb{S}_{\bar\kappa} \leq c_n(\bar\kappa)-\epsilon_h\} \cup \bigl( \mathcal{B}' \cap \{\mathbb{S}_{\bar\kappa} \leq c_n(\bar\kappa)+ (C+1)\epsilon_h\} \bigr),
\end{equation}
for some bounded subset $\mathcal{B}'$ of $\mathcal{M}$. Since $\mathbb{S}_{\bar\kappa}$ satisfies the Palais-Smale condition on bounded subsets of $\mathcal{M}$ on which $\mathbb{S}_{\bar\kappa}$ is bounded away from zero (see \cite[Lemmata 5.1 and 5.3]{abb13}), the set 
\[
K:= \mathcal{B}' \cap \mathrm{crit} (\mathbb{S}_{\bar\kappa}) \cap \mathbb{S}_{\bar\kappa}^{-1}(c_n(\bar\kappa))
\]
is compact. Since $K$ consists of fixed points of the flow $\phi$, it has a neighborhood  $\mathcal{V}'\subset \mathcal{V}$ such that 
\begin{equation}
\label{piccolo}
\phi([0,1]\times \mathcal{V}') \subset \mathcal{V}.
\end{equation}
Using again the fact that  $\mathbb{S}_{\bar\kappa}$ satisfies the Palais-Smale condition on bounded subsets of $\mathcal{M}$ on which $\mathbb{S}_{\bar\kappa}$ is bounded away from zero, we can find $\epsilon>0$ and $\delta\in (0,1]$ such that
\begin{equation}
\label{daps}
\|d \mathbb{S}_{\bar\kappa}\| \geq \delta \qquad \mbox{on} \quad \bigl( \mathcal{B}' \setminus \mathcal{V}' \bigr) \cap \{ c_n(\bar{\kappa}) - \epsilon \leq  \mathbb{S}_{\bar\kappa} \leq c_n(\bar{\kappa}) + \epsilon\}.
\end{equation}
Let $\sigma\in [0,1]$ be such that
\begin{equation}
\label{impo}
\mathbb{S}_{\bar\kappa}(\phi_1(v_h(\sigma))) \geq c_n(\bar{\kappa}) \qquad \mbox{and} \qquad \phi_1(v_h(\sigma)) \notin \mathcal{V}.
\end{equation}
By (\ref{piccolo}), $\phi_r(v_h(\sigma))$ cannot belong to $\mathcal{V}'$ for any $r\in [0,1]$. Together with (\ref{dovephi}) we deduce that
\[
\phi([0,1]\times \{v_h(\sigma)\}) \subset \bigl( \mathcal{B}' \setminus \mathcal{V}' \bigr) \cap \{ c_n(\bar{\kappa}) \leq  \mathbb{S}_{\bar\kappa} \leq c_n(\bar{\kappa}) + (C+1)\epsilon_h\}.
\]
When $h$ is so large that $(C+1)\epsilon_h \leq \epsilon$, (\ref{daps}) implies that
\[
\|d \mathbb{S}_{\bar\kappa}(\phi_r(u_h(\sigma)))\| \geq \delta, \qquad \forall r\in [0,1],
\]
and by (\ref{decr}) we find
\[
\begin{split}
c_n(\bar\kappa) &\leq  \mathbb{S}_{\bar\kappa}(\phi_1(v_h(\sigma)))\\
& = \mathbb{S}_{\bar\kappa}(v_h(\sigma)) + \int_0^1 \frac{d}{dr} \mathbb{S}_{\bar\kappa}(\phi_r(v_h(\sigma)))\, dr \\ &\leq c_n(\bar\kappa)+ (C+1)\epsilon_h + \int_0^1 d\mathbb{S}_{\bar\kappa}(\phi_r(v_h(\sigma)))[X(\phi_r(v_h(\sigma)))]\, dr \\ &\leq c_n(\bar\kappa) + (C+1)\epsilon_h - \delta^2.
\end{split}
\]
Since $(\epsilon_h)$ is infinitesimal, the above inequality implies that $h$ is smaller than some $h_0$. When $h$ is larger than $h_0$, (\ref{impo}) cannot occur and hence
\[
\phi_1(v_h([0,1])) \subset \{\mathbb{S}_{\bar\kappa}<c_n(\bar\kappa)\} \cup \mathcal{V},
\]
as claimed.
\end{proof}

\subsection{The proof of the theorem}\label{s:proof}

We can finally prove the theorem stated in the Introduction. Let $\kappa\in (0,c_u)$.  If the local minimizer $\alpha_{\kappa}$ is not strict, then $\T\cdot \alpha_{\kappa}$ is the limit of a sequence of critical circles of $\mathbb{S}_{\kappa}$ in $\mathcal{M} \setminus (\T\cdot \alpha_{\kappa})$, which determine infinitely many distinct periodic orbits of energy $\kappa$. Therefore, it is enough to prove the following statement: 

\medskip

\noindent {\em Every $\kappa_*$ in $(0,c_u)$ such that the local minimizer $\alpha_{\kappa_*}$ is strict has a neighborhood $J\subset (0,c_u)$ such that for almost every $\kappa\in J$ the energy level $E^{-1}(\kappa)$ has infinitely many periodic orbits.} 

\medskip

Fix $\kappa_*$ as above and let $J=J(\kappa_*)\subset (0,c_u)$ be the interval which is constructed in Section \ref{minimaxsec}. Let $c_n:J \rightarrow \R$ be the corresponding sequence of minimax functions and let $n_0\in \N$ be such that $c_n<0$ for every $n\geq n_0$. Since the countably many functions $c_n$ are monotone, Lebesgue's theorem implies that the set
\[
J':= \set{\kappa\in \mathrm{Int}(J)}{c_n \mbox{ is differentiable at $\kappa$ for every $n\geq n_0$}}
\]
has full measure in $J$. We shall prove that for every $\kappa\in J'$ the energy level $E^{-1}(\kappa)$ has infinitely many periodic orbits.

Fix some $\kappa\in J'$. If $M_{\kappa}$ consists of infinitely many critical circles, then $E^{-1}(\kappa)$ has clearly infinitely many periodic orbits. Therefore, we may assume that $M_{\kappa}$ consists of only finitely many critical circles, which thus consist of strict local minimizers.
Assume by contradiction that $E^{-1}(\kappa)$ has only finitely many periodic orbits. Then the critical set of $\mathbb{S}_{\kappa}$ consists of finitely many critical circles
\[
\T\cdot \gamma_1, \; \T\cdot \gamma_2, \dots, \T\cdot \gamma_k
\]
together with their iterates $\T\cdot \gamma_j^n$, for $1\leq j\leq k$ and $n\in \N$. By Theorem \ref{t:iterated_mountain_pass} we can find a natural number $n_1$ such that the following is true: For every $n\geq n_1$ and for every $j\in \{1,\dots,k\}$ there exists a neighborhood $\mathcal{W}_{j,n}$ of $\T\cdot \gamma_j^n$ such that any two points in $\{\mathbb{S}_{\kappa}<\mathbb{S}_{\kappa}(\gamma_j^n)\}$ which can be connected within 
\[
\{\mathbb{S}_{\kappa}<\mathbb{S}_{\kappa}(\gamma_j^n)\}\cup \mathcal{W}_{j,n}
\]
can be also connected in $\{\mathbb{S}_{\kappa}<\mathbb{S}_{\kappa}(\gamma_j^n)\}$. Moreover, the sets $\mathcal{W}_{j,n}$ can be chosen to be so small that their closures are pairwise disjoint. Set
\[
a:= \min_{1\leq j \leq k} \mathbb{S}_{\kappa}(\gamma_j^{n_1-1}).
\]
By (\ref{neg}) we can find a natural number $n\geq n_0$ such that $c_n(\kappa)<a$. By Lemma \ref{struwe}, $c_n(\kappa)$ is a critical value of $\mathbb{S}_{\kappa}$, and by our finiteness assumption, 
\[
\mathrm{crit} (\mathbb{S}_{\kappa}) \cap \mathbb{S}_{\kappa}^{-1}(c_n(\kappa)) = \T\cdot \gamma_{j_1}^{m_1} \cup \dots \cup \T\cdot \gamma_{j_h}^{m_h},
\]
for some non-empty subset $\{j_1,\dots,j_h\}$ of $\{1,\dots,k\}$ and some positive integers $m_1,\dots,m_h$. Since $c_n(\kappa)<a$, all the $m_i$'s are at least $n_1$. We apply Lemma \ref{struwe} with 
\[
\mathcal{V} := \mathcal{W}_{j_1,m_1} \cup \dots \cup \mathcal{W}_{j_h,m_h}
\] 
and we obtain a path $u\in \mathcal{P}_n(\kappa)$ with image in 
\[
\{\mathbb{S}_{\kappa}< c_n(\kappa)\} \cup \mathcal{V},
\]
and such that $\mathbb{S}_{\kappa}(u(0))<c_n(\kappa)$. Since also $\mathbb{S}_{\kappa}(u(1)) = \mathbb{S}_{\kappa}(\mu^n)<c_n(\kappa)$ by (\ref{prima}), and since the sets $\mathcal{W}_{j_i,m_i}$, $1\leq i \leq h$, have pairwise disjoint closures, the path $u$ is the concatenation of finitely many paths $v$, each of which has end-points in $\{\mathbb{S}_{\kappa}< c_n(\kappa)\}$ and is contained in $\{\mathbb{S}_{\kappa}< c_n(\kappa)\} \cup \mathcal{W}_{j_i,m_i}$ for some $i\in \{1,\dots,h\}$. By the property of the sets $\mathcal{W}_{j_i,m_i}$ stated above, the end-points of each of the $v$'s can be joined by paths $w$ in $\{\mathbb{S}_{\kappa}< c_n(\kappa)\}$. By concatenating the $w$'s, we obtain a path in $\{\mathbb{S}_{\kappa}< c_n(\kappa)\}$ which joins 
$u(0)$ and $u(1)$. Since such a path belongs to $\mathcal{P}_n(\kappa)$, this contradicts the definition of $c_n(\kappa)$. This contradiction proves that $E^{-1}(\kappa)$ has infinitely many periodic orbits.

\appendix
\section{Hyperbolic perturbation of unipotent symplectic matrices}

We recall that an element of $\mathrm{GL}(n,\R)$ is called unipotent when its spectrum  is equal to $\{1\}$, while it is called  hyperbolic when its spectrum does not intersect the unit circle of the complex plane. The argument in the proof of the following statement was  suggested to us by Marie-Claude Arnaud and Jairo Bochi. Notice that the statement becomes straightforward if $n=1$, since $\mathrm{Sp}(2)=\mathrm{SL}(2,\R)$.

\begin{prop}\label{p:hyperbolic_perturbation}
The unipotent elements of $\mathrm{Sp}(2n)$ belong to the closure of the space of hyperbolic elements of $\mathrm{Sp}(2n)$.
\end{prop}

\begin{proof}
Consider an arbitrary $P\in\mathrm{Sp}(2n)$ with spectrum  $\sigma(P)=\{1\}$. Let us prove that there exists a Lagrangian vector subspace $V\subset\R^{2n}$ invariant by $P$. In order to do this, let $V$ be a maximal isotropic subspace invariant by $P$, that is, an invariant isotropic subspace that is not strictly contained in another isotropic invariant subspace. We claim that $V$ is Lagrangian. Indeed, arguing by contradiction, suppose that $\dim V < n$. Let $V^\omega$ be the symplectic orthogonal of $V$ and notice that $P(V^\omega) = V^\omega$. By our hypothesis, $V$ is strictly contained in $V^\omega$ and therefore has a positive-dimensional  complementary subspace $V'$ in $V^\omega$. Using the decomposition $V^\omega = V \oplus V'$ we can write
\[
P|_{V^\omega} = \left(
\begin{matrix}
P|_V &  \pi_1\circ P|_{V'} \\
0 & \pi_2\circ P|_{V'}
\end{matrix} \right),
\]
where $\pi_1: V^\omega \to V$ and $\pi_2: V^\omega \to V'$ are the projections. Since $\sigma(P|_{V^\omega})=\sigma(P|_{V})=\{1\}$ and
\[
\det (P|_{V^\omega}-\lambda I) = \det (P|_{V}-\lambda I) \det (\pi_2\circ P|_{V'}-\lambda I),\qquad\forall\lambda\in\C,
\]
we infer that $\sigma(\pi_2\circ P|_{V'})=\{1\}$. Consequently, there exists $v \in V'$ such that $Pv-v \in V$. But this implies that $V \oplus \text{span}\{v\}$ is isotropic and invariant, contradicting the maximality of $V$.

Now, choose a symplectic basis $\{e_1,\dots,e_n,f_1,\dots,f_n\}$ such that $V=\text{span}\{e_1,\dots,e_n\}$. Using these coordinates we can write $P$ as
\[
P = \left(
\begin{matrix}
A & B \\
0 & C
\end{matrix} \right),
\]
where $A = P|_V$. For all $t\in\R$, we set
\[
P_{t}:=
\left(
\begin{matrix}
e^t I& 0 \\
0 & e^{-t} I
\end{matrix} \right)\cdot\left(
\begin{matrix}
A & B \\
0 & C
\end{matrix} \right)
=
\left(
\begin{matrix}
e^t A & e^t B \\
0 & e^{-t}C
\end{matrix} \right),
\]
Notice that $P_{t}\in\mathrm{Sp}(2n)$, being the product of two elements of $\mathrm{Sp}(2n)$. Clearly, $P_t$ depends smoothly on $t$, and  $P_0=P$. Finally, since $\sigma(A)=\{1\}$, we have that $e^t$ is an eigenvalue of $P_t$ with algebraic multiplicity $n$. This, together with the fact that $P_t$ is symplectic, implies that for all $t\neq0$ the matrix $P_t$ is hyperbolic with $\sigma(P_t) = \{e^t,e^{-t}\}$.
\end{proof}

\bibliographystyle{amsalpha}
\bibliography{nonlinear}

\end{document}